\documentclass[12pt,a4paper,leqno,twoside]{amsart}
\usepackage[a4paper,left=3cm,top=4cm,bottom=4cm,right=3cm]{geometry}
\usepackage[T1]{fontenc}
\usepackage{amssymb,bbm}
\usepackage{amsmath,amsthm,dsfont}
\usepackage{amsfonts,mathrsfs,wasysym}
\usepackage{latexsym}
\usepackage[cp1250]{inputenc}
\usepackage{graphicx}
\usepackage{indentfirst}
\usepackage{enumerate}
\usepackage{geometry}
\usepackage{enumitem}
\usepackage{xparse}
\usepackage{etoolbox}
\usepackage{mathptmx}
\usepackage{fancyhdr}
\usepackage[usenames,dvipsnames]{xcolor}
\usepackage{marvosym}
\usepackage{pifont}
\usepackage{cite}

\usepackage{hyperref}
\hypersetup{colorlinks,linkcolor={blue},citecolor={red}}

\patchcmd{\abstract}{\scshape\abstractname}{\textbf{\abstractname}}{}{}

\makeatletter
\DeclareMathAlphabet{\mathcal}{OMS}{cmsy}{m}{n}

\DeclareSymbolFont{operators}{OT1}{ztmcm}{m}{n}
\DeclareSymbolFont{letters}{OML}{ztmcm}{m}{it}
\DeclareSymbolFont{symbols}{OMS}{ztmcm}{m}{n}
\DeclareSymbolFont{largesymbols}{OMX}{ztmcm}{m}{n}
\DeclareSymbolFont{bold}{OT1}{ptm}{bx}{n}
\DeclareSymbolFont{italic}{OT1}{ptm}{m}{it}
\@ifundefined{mathbf}{}{\DeclareMathAlphabet{\mathbf}{OT1}{ptm}{bx}{n}}
\@ifundefined{mathit}{}{\DeclareMathAlphabet{\mathit}{OT1}{ptm}{m}{it}}
\DeclareMathSymbol{\omicron}{0}{operators}{`\o}

\DeclareMathAlphabet{\mathpzc}{OT1}{pzc}{m}{it}
\DeclareSymbolFont{operators}{OT1}{txr}{m}{n}
\SetSymbolFont{operators}{bold}{OT1}{txr}{bx}{n}
\def\operator@font{\mathgroup\symoperators}
\DeclareSymbolFont{italic}{OT1}{txr}{m}{it}
\SetSymbolFont{italic}{bold}{OT1}{txr}{bx}{it}
\DeclareSymbolFontAlphabet{\mathrm}{operators}
\DeclareMathAlphabet{\mathbf}{OT1}{txr}{bx}{n}
\DeclareMathAlphabet{\mathit}{OT1}{txr}{m}{it}
\SetMathAlphabet{\mathit}{bold}{OT1}{txr}{bx}{it}
\DeclareSymbolFont{letters}{OML}{txmi}{m}{it}
\SetSymbolFont{letters}{bold}{OML}{txmi}{bx}{it}
\DeclareFontSubstitution{OML}{txmi}{m}{it}
\DeclareSymbolFont{lettersA}{U}{txmia}{m}{it}
\SetSymbolFont{lettersA}{bold}{U}{txmia}{bx}{it}
\DeclareFontSubstitution{U}{txmia}{m}{it}
\DeclareSymbolFontAlphabet{\mathfrak}{lettersA}
\DeclareSymbolFont{symbols}{OMS}{txsy}{m}{n}
\SetSymbolFont{symbols}{bold}{OMS}{txsy}{bx}{n}
\DeclareFontSubstitution{OMS}{txsy}{m}{n}
\makeatother

\makeatletter
\renewcommand\abstractname{\scshape\bfseries Abstract}

\renewenvironment{proof}[1][\proofname]{\par \pushQED{\qed} \normalfont
  \topsep6\p@\@plus6\p@ \trivlist \itemindent\z@
  \item[\hskip\labelsep\bfseries
    #1\@addpunct{.}]\ignorespaces
}{
  \popQED\endtrivlist\@endpefalse
}

\makeatother

\makeatletter
    \@addtoreset{equation}{section} 
    \renewcommand{\theequation}{{\thesection}.\@arabic\c@equation} 

\makeatother

\makeatletter

\def\section{\@ifstar\unnumberedsection\numberedsection}
\def\numberedsection{\@ifnextchar[
  \numberedsectionwithtwoarguments\numberedsectionwithoneargument}
\def\unnumberedsection{\@ifnextchar[
  \unnumberedsectionwithtwoarguments\unnumberedsectionwithoneargument}
\def\numberedsectionwithoneargument#1{\numberedsectionwithtwoarguments[#1]{#1}}
\def\unnumberedsectionwithoneargument#1{\unnumberedsectionwithtwoarguments[#1]{#1}}
\def\numberedsectionwithtwoarguments[#1]#2{%
  \ifhmode\par\fi
  \removelastskip
  \vskip 4ex\goodbreak
  \refstepcounter{section}%
  \noindent
  \begingroup
  \leavevmode\centering\scshape\bfseries
  \thesection.
  #2
  \par
  \endgroup
  \vskip 1ex\nobreak
  \addcontentsline{toc}{section}{%
    \protect\numberline{\thesection}%
    #1}%
  }
\def\unnumberedsectionwithtwoarguments[#1]#2{%
  \ifhmode\par\fi
  \removelastskip
  \vskip 2ex\goodbreak
  \noindent
  \begingroup
  \leavevmode\centering\scshape\bfseries
  \leavevmode\centering\scshape\bfseries
  #2
  \par
  \endgroup
  \vskip 1ex\nobreak
  \addcontentsline{toc}{section}{%
    #1}%
}
\makeatother

\makeatletter
\def\@seccntformat#1{\csname mythe#1\endcsname}

\let\latex@subsection\subsection
\def\subsection{\@ifstar{\refstepcounter{subsection}\latex@subsection*}{\latex@subsection}}
\def\@makechapterhead#1{%
  \vspace*{40\p@}%
  {\parindent \z@ \raggedright \normalfont
    \interlinepenalty\@M
    \Huge \bfseries #1\par \nobreak
    \vskip 40\p@
  }}
\let\latex@l@chapter\l@chapter
\def\l@chapter#1#2{\begingroup\let\numberline\@gobble\latex@l@chapter{#1}{#2}\endgroup}
\makeatother

\theoremstyle{plain}
\newtheorem{Th}{Theorem}[section]
\newtheorem{Prop}[Th]{Proposition}
\newtheorem{Lem}[Th]{Lemma}
\newtheorem{Cor}[Th]{Corollary}
\theoremstyle{definition}
\newtheorem{Rem}[Th]{Remark}

\newtheorem{Def}[Th]{Definition}

\def\bf{\textbf}
\def\it{\textit}
\def\te{\textnormal}
\def\tn{\textnormal}

\def\leq{\leqslant}
\def\geq{\geqslant}
\def\R{{\mathds R}}
\def\N{{\mathds N}}

\def\B{\textit{I\!B}}

\def\D{{\mathrm{dom}}\,}
\def\E{{\mathrm{epi}}\,}

\def\G{{\mathrm{gph}}\,}

\def\e{\mathsf{e}}
\def\kr{\bar}

\begin{document}
\title{Representation of Hamilton-Jacobi equation in optimal control theory with  unbounded control set}


\author{\vspace*{-0.2cm}{Arkadiusz Misztela \textdagger}\vspace*{-0.2cm}}
\thanks{\textdagger\, Institute of Mathematics, University of Szczecin, Wielkopolska 15, 70-451 Szczecin, Poland; e-mail: arkadiusz.misztela@usz.edu.pl}


\begin{abstract} 
In this paper we study the existence of sufficiently regular representations of  Hamilton-Jacobi  equations in the optimal control theory with unbounded control set.  We use a new method to construct representations for a wide class of Hamiltonians. This class is wider than  any constructed before, because we do not require Legendre-Fenchel conjugates of Hamiltonians to be bounded.  However, in this case we obtain representations with unbounded control set. We apply the obtained results to study regularities of value functions and correlations between variational and optimal control problems. \\
\vspace{0mm}

\hspace{-1cm}
\hspace*{-0.55cm}  theory,  parametrization of set-valued maps, convex analysis.

\vspace{3mm}\hspace{-1cm}
\end{abstract}

\maketitle

\pagestyle{myheadings}  \markboth{\small{\scshape Arkadiusz Misztela}
}{\small{\scshape Representations of Hamiltonians}}

\thispagestyle{empty}
%

\vspace{-1cm}



\section{Introduction}

\noindent The Hamilton-Jacobi equation 
\begin{equation}\label{rowhj}
\begin{array}{rll}
-V_{t}+ H(t,x,-V_{x})=0 &\;\tn{in}\;& (0,T)\times\R^n, \\
V(T,x)=g(x) &\;\tn{in}\;&\R^n,
\end{array}
\end{equation}
with a  convex Hamiltonian $H$  in the gradient variable can be studied with connection to 
calculus of variations problems. Let  $H^{\ast}$ be the Legendre-Fenchel conjugate of $H$ in its gradient variable:
\begin{equation*}\label{tran1}
H^{\ast}(t,x,v)= \sup_{p\in\R^{n}}\,\{\,\langle v,p\rangle-H(t,x,p)\,\}.
\end{equation*}
Then the value function of the calculus of variations problem defined by 
\begin{equation}\label{fwwp}
V(t_0,x_0)= \inf_{\begin{array}{c}
\scriptstyle x(\cdot)\,\in\,\mathcal{A}([t_0,T],\R^n)\\[-1mm]
\scriptstyle x(t_0)=x_0
\end{array}}\,\big\{\,g(x(T))+\int_{t_0}^TH^{\ast}(t,x(t),\dot{x}(t))\,dt\,\big\}
\end{equation}
is the unique viscosity solution of \eqref{rowhj}; see  \cite{C-S-2004,DM-F-V,HF,F-P-Rz,AM2,AM0,P-Q}, where $\mathcal{A}([t_0,T],\R^n)$ denotes the space of all absolutely continuous functions from $[t_0,T]$ into $\R^n$.

\vspace{2mm}
The Hamilton-Jacobi equation \eqref{rowhj}  can be also studied with connection to optimal control problems. It is possible, provided that there exists  a sufficiently regular triple $(A,f,l)$ satisfying the following equality
\begin{equation}\label{hfl}
H(t,x,p)=  \sup_{a\in A}\,\{\,\langle\, p\,,f(t,x,a)\,\rangle\,-\,l(t,x,a)\,\}.
\end{equation}
Then the value function of the optimal control problem defined by 
\begin{equation}\label{fwfl}
V(t_0,x_0)= \inf_{(x,\,a)(\cdot)\,\in\, \emph{S}_f(t_0,x_0)}\,\big\{\,g(x(T))+\int_{t_0}^Tl(t,x(t),a(t))\,dt\,\big\}
\end{equation}
is the unique viscosity solution of \eqref{rowhj}; see  \cite{C-S-2004,HF,F-P-Rz,B-CD,FR0}, where $\emph{S}_f(t_0,x_0)$ denotes the set of all trajectory-control pairs of the control system
\begin{equation}\label{scs}
\begin{array}{ll}
\dot{x}(t)=f(t,x(t),a(t)),& a(t)\in A\;\;\mathrm{a.e.}
\;\;t\in[t_0,T],\\
x(t_0)=x_0.&
\end{array}
\end{equation}
While working with optimal control problems we require $f$ and $l$ to be locally Lipschitz continuous functions with respect to $x$. In addition to this $f$ is to have the sublinear growth with respect to $x$ and $l$ is to have the sublinear growth with respect to $a$. This guarantees that to every integrable control $a(\cdot)$ on $[t_0,T]$ with values in a closed  subset $A$ of $\R^m$  there corresponds the unique solution $x(\cdot)$ of \eqref{scs} defined on $[t_0,T]$ and $l(\cdot,x(\cdot),a(\cdot))$ is an integrable function on $[t_0,T]$. 

\vspace{2mm}
The triple $(A,f,l)$, which satisfies the equality \eqref{hfl},  is called a  \it{representation} of $H$.
In general, if a representation of $H$ exists, then infinitely many other representations exist. There are also irregular representations among them. The triple $(A,f,l)$, which satisfies the equality \eqref{hfl} and inherits Lipschitz-type properties of the Hamiltonian $H$, is called a \it{faithful representation} of the Hamiltonian $H$

\vspace{2mm}
In this paper we provide further developments of representation theorems from  \cite{AM}. Misztela \cite{AM} studied  faithful representations of Hamiltonians with the compact control set. A necessary condition for the existence of such representations is boundedness of  Legendre-Fenchel conjugates of  Hamiltonians on effective domains; see \cite[Thm. 3.1]{AM}. However, in many cases Hamiltonians  do not have bounded  Legendre-Fenchel conjugates on effective domains. In Section \ref{section-3} we see that for this type of Hamiltonians there exist faithful representations with the unbounded control set. We used a new method to construct a faithful representation. 
Our representation $(A,f,l)$ of $H$ is an epigraphical representation, i.e. a triple
$(A,f,l)$ which satisfies the following condition 
\begin{equation}\label{crep}
\G H^{\ast}(t,x,\cdot)\subset(f(t,x,A),l(t,x,A))\subset\E H^{\ast}(t,x,\cdot),
\end{equation}
where $(f(t,x,A),l(t,x,A))$ denotes the set $\{(f(t,x,a),l(t,x,a))\mid a\in A\}$. The construction of this representation is as follows: first, using the Steiner Selection we parametrize the set $\E H^{\ast}(t,x,\cdot)$ in such a way that $\e(t,x,A)=\E H^{\ast}(t,x,\cdot)$. The Steiner Selection guarantees that $\e$ is  local Lipschitz continuous with respect to $x$. Next, we define the functions $f$, $l$ as components of the function $\e$, i.e. $\e=(f,l)$. From the equality $\e(t,x,A)=\E H^{\ast}(t,x,\cdot)$ it follows that \eqref{crep} holds. In view of [10, Prop. 5.7] any triple $(A,f,l)$ satisfying  \eqref{crep} is a representation of $H$.
Earlier, Frankowska-Sedrakyan \cite{F-S} and
Rampazzo \cite{FR} used a graphical representation to construct a faithful representation. The representation $(A,f,l)$ of $H$ is a graphical representation, if a triple $(A,f,l)$ satisfies  $(f(t,x,A),l(t,x,A))=\G H^{\ast}(t,x,\cdot)$. In a graphical representation the function $l$, without additional assumptions on $H^{\ast}$, may be discontinuous with respect to $(x,a)$; see Section~\ref{dfrentger}. Another differences between graphical and epigraphical representations can be found in \cite{AM}. Earlier,  Ishii \cite{HI} proposed a representation involving continuous functions $f$, $l$ with the infinite-dimensional control set $A$. The lack of local Lipschitz continuity of  $f$ and $l$ with respect to  $x$ and finite-dimensional control set $A$ in Ishii \cite{HI} paper causes troubles in applications. 

\vspace{2mm}
We present differences  between  representations with unbounded and compact control sets. The fact that a control set is not compact makes significant problems in applications which we discuss below. Therefore, compactness of a control set must be replaced by  another property that is convenient in practice. The following property which is a consequence of our construction of a faithful representation plays a role of such  extra-property: 
\begin{equation*}
a=(f(t,x,a),l(t,x,a))\quad\tn{for all}\quad a\in\E H^{\ast}(t,x,\cdot).
\end{equation*}
Our extra-property is apparently new. In literature one usually requires coercivity of the function  $l(t,x,\cdot)$; see, e.g. \cite[Condition ($A_4$)]{FR0}. However, the function $l(t,x,\cdot)$ from our faithful representation $(A,f,l)$ does not have this property.  Coercivity of the function  $l(t,x,\cdot)$  enables us to study not only measurability of controls but also its integrability.  In this paper the extra-property plays a similar role; see Remarks \ref{extraproperty} and \ref{cor-bvocp-rem}. It is well-known that in applications one requires at least integrability of controls. In the case when the control set is compact the above problem does not occur, because every measurable control with values in
the compact control set is integrable.

\vspace{2mm}
In general, the value functions \eqref{fwwp} and \eqref{fwfl} are not equal. However, in our case these value functions are identical due to the extra-property; see Corollary \ref{cor-reduct}. Moreover, we obtain a  fundamental relation between variational and optimal control problems; see Theorem \ref{thm-reduct}. More precisely, we consider a variational problem associated with the given Lagrangian $L$. We define Hamiltonian $H$ as  the Legendre-Fenchel transform of $L$ in its velocity variable. Applying our result to Hamiltonian $H$ we obtain its faithful representation $(A,f,l)$. Then the variational  problem associated with Lagrangian $L$ is equivalent to the optimal control problem associated with the triple $(A,f,l)$. Earlier, Olech \cite{CO-69} and Rockafeller \cite{RTR73,RTR} investigated the opposite problem. They considered the optimal control problem associated with the given triple  $(A,f,l)$. Using this triple they defined Lagrangian $L$ in such a way that  the optimal control problem associated with the triple  $(A,f,l)$ is equivalent to the variational problem associated with Lagrangian $L$.  

\vspace{2mm}
Our faithful representations are stable; see Theorems \ref{cor-rep-stab2} and \ref{cor-rep-stab4}. This fact is used in the proof of stability of value functions; see Section \ref{thm-reguvalufun}. The method of this proof is not standard, because  properties of a faithful representation are nonstandard.  These nonstandard properties are unbounded control set, the extra-property and the sublinear growth of  $l$ with respect to $a$. In this case one cannot apply methods from Sedrakyan \cite{HS} to prove stability of value functions. Indeed, this method 
uses compactness of the control set and boundedness of $l$ independent of $a$. We also prove that the value function $V$ is locally Lipschitz continuous, provided that $g$ is locally Lipschitz continuous. In the proof of this fact nonstandard boundedness of the function $l$ plays significant role.

\vspace{2mm}
The outline of the paper is as follows. Section \ref{section-2} contains hypotheses and preliminary results. In Section \ref{dfrentger} we show differences between graphical and epigraphical representations with the unbounded control set. In Section~\ref{section-3} we gathered our main results. Sections \ref{thm-profsrst}  and \ref{thm-reguvalufun} contain proofs of results from Section~\ref{section-3}. Section \ref{conrem-conclude} contains concluding remarks.



\section{Hypotheses and background material}\label{section-2}

\noindent We will need hypotheses and results similar to those in \cite[Sect. 2]{AM}.

\begin{enumerate}[leftmargin=9.7mm]
\item[\tn{\bf{(H1)}}] $H:[0,T]\times\R^{n}\times\R^{n}\rightarrow\R$
is Lebesgue measurable  in $t$ for any $x,p\in\R^n$;
\item[\tn{\bf{(H2)}}] $H(t,x,p)$ is continuous with respect to $(x,p)$ for every $t\in[0,T]$;
\item[\tn{\bf{(H3)}}] $H(t,x,p)$ is convex with respect to $p$ for every $(t,x)\in[0,T]\times\R^n$;
\item[\tn{\bf{(H4)}}] There exists a measurable map $c:[0,T]\to[0,+\infty)$ such that for every\\ 
\hspace*{-8.5mm}$t\in[0,T]$ and $x,p,q\in\R^n$ one has $|H(t,x,p)-H(t,x,q)|\leq c(t)(1+|x|)|p-q|$.
\end{enumerate}

Let $\varphi$ be an extended-real-valued function. The sets: $\D\varphi=\{\,z\in\R^m\mid\varphi(z)\not=\pm\infty\,\}$, $\G\varphi=\{\,(z,r)\in\R^m\times\R\mid\varphi(z)=r\,\}$ and $\E\varphi=\{\,(z,r)\in\R^m\times\R\mid\varphi(z)\leq r\,\}$ are called the \emph{effective domain}, the \it{graph} and the \it{epigraph} of $\varphi$, respectively. We say that $\varphi$ is \it{proper} if it never takes  the value $-\infty$ and it is not identically equal to $+\infty$.  Using properties of the Legendre-Fenchel conjugate from~\cite{R-W} we can prove the following proposition.

\begin{Prop}\label{prop2-fmw} Assume that $H$ satisfies \tn{(H1)-(H3)}. Then
\begin{enumerate}
\item[\tn{\bf{(C1)}}] $H^{\ast}:[0,T]\times\R^{n}\times\R^{n}\rightarrow\R\cup\{+\infty\}$ is Lebesgue-Borel-Borel measurable\te{;}
\item[\tn{\bf{(C2)}}] $H^{\ast}(t,x,v)$  is  lower semicontinuous  with respect to $(x,v)$ for every $t\in[0,T]$\tn{;}
\item[\tn{\bf{(C3)}}] $H^{\ast}(t,x,v)$ is convex and proper with respect to $v$ for every $(t,x)\in[0,T]\times\R^n$\tn{;}
\item[\tn{\bf{(C4)}}] $\forall\,(t,x,v)\in[0,T]\times\R^n\times\R^n\;\;\forall\,x_i\rightarrow x\;\;
\exists\,v_i\rightarrow v\;:\;H^{\ast}(t,x_i,v_i)\rightarrow H^{\ast}(t,x,v)$\tn{.}
\item[]\hspace{-1.3cm}Additionally, if $H$ satisfies \tn{(H4)}, then
\item[\tn{\bf{(C5)}}] $\forall\,(t,x,v)\in[0,T]\times\R^n\times\R^n \;:\; |v|>c(t)(1+|x|)\;\Rightarrow\; H^{\ast}(t,x,v) =+\infty$\tn{.}
\item[]\hspace{-1.3cm}Additionally, if $H$ is continuous, then $H^{\ast}$ is lower semicontinuous and
\item[\tn{\bf{(C6)}}] $\forall\,(t,x,v)\in[0,T]\times\R^n\times\R^n\;\;\forall\,(t_i,x_i)\rightarrow (t,x)\;\;\exists\,v_i\rightarrow v\;:\;H^{\ast}(t_i,x_i,v_i)\rightarrow H^{\ast}(t,x,v)$.
\end{enumerate}
\end{Prop}

Let $K$ be a nonempty subset of $\R^m$. We put $\|K\|:=\sup_{\xi\in K}|\xi|$. The distance from $y\in\R^m$ to $K$ is defined by $d(y,K):=\inf_{\xi\in K}|y-\xi|$.  A set-valued map $F:[0,T]\multimap\R^m$ is \it{measurable} if for each open set $U\subset\R^m$ the inverse image  $F^{-1}(U):= \{\,t\in[0,T]\mid F(t)\cap U\not=\emptyset\,\}$
is Lebesgue measurable set. The set $\G F:= \{\,(z,y)\mid y\in F(z)\,\}$ is called a \it{graph} of the set-valued map $F$. A set-valued map $F:\R^n\multimap\R^m$ is \it{lower semicontinuous} in  Kuratowski's sense if for each open set $U\subset\R^m$ the set $F^{-1}(U)$ is open.  

\vspace{2mm}
Let us define the set-valued map $E_{H^{\ast}}:[0,T]\times\R^n\multimap\R^n\times\R$ by the formula
\begin{equation*}
E_{H^{\ast}}(t,x):=\E H^{\ast}(t,x,\cdot)= \{\,(v,\eta)\in\R^n\times\R\,\mid\, H^{\ast}(t,x,v)\leq \eta\,\}.
\end{equation*}

\vspace{2mm}
From Proposition \ref{prop2-fmw} and Results in \cite[Chap. 14]{R-W} we deduce the following corollary.

\begin{Cor}\label{wrow-wm}
Assume that $H$ satisfies \tn{(H1)-(H3)}. Then
\begin{enumerate}
\item[\tn{\bf{(E1)}}] $E_{H^{\ast}}(t,x)$ is a nonempty, closed, convex subset of $\;\R^{n+1}$ for all $(t,x)\in[0,T]\times\R^n$\tn{;}
\item[\tn{\bf{(E2)}}] $x\to E_{H^{\ast}}(t,x)$ has a closed graph for every $t\in[0,T]$\tn{;}
\item[\tn{\bf{(E3)}}] $x\to E_{H^{\ast}}(t,x)$ is lower semicontinuous for every  $t\in[0,T]$\tn{;}
\item[\tn{\bf{(E4)}}] $t\to E_{H^{\ast}}(t,x)$ is measurable for every $x\in\R^n$\tn{.}
\item[]\hspace{-1.3cm}Additionally, if $H$ satisfies \tn{(H4)}, then
\item[\tn{\bf{(E5)}}] $\|\D {H^{\ast}}(t,x,\cdot)\|\leq c(t)(1+|x|)$ for every $(t,x)\in[0,T]\times\R^n$\tn{.}
\item[]\hspace{-1.3cm}Additionally, if  $H$ is continuous, then
\item[\tn{\bf{(E6)}}] $(t,x)\to E_{H^{\ast}}(t,x)$ has a closed graph and is lower semicontinuous.
\end{enumerate}
\end{Cor}

Now we present Hausdorff continuity of a set-valued map $E_{H^{\ast}}$ in Hamiltonian and its conjugate terms. Let $\B(\kr{x},R)$ denote the closed ball in $\R^n$ of center $\kr{x}$ and radius $R\geq 0$. We put $\B_R:=\B(0,R)$ and $\B:=\B(0,1)$. 

\begin{Th}[\tn{\cite[Thm. 2.3]{AM}}]\label{tw2_rlhmh}
Assume that $H$ satisfies \tn{(H1)-(H3)}. Then the following conditions are equivalent:

\tn{\bf{(HLC)}} For any $R>0$ there exists a measurable map $k_R:[0,T]\to[0,+\infty)$ such that $|\,H(t,x,p)-H(t,y,p)\,|\,\leq\, k_R(t)\,(1+|p|)\,|x-y|$ for all $t\in[0,T]$, $x,y\in\B_R$, $p\in\R^n$.

\tn{\bf{(CLC)}} For any $R>0$ there exists a measurable map $k_R:[0,T]\to[0,+\infty)$  such that for all $t\in[0,T]$, $x,y\in \B_R$, $v\in\D H^{\ast}(t,x,\cdot)$ there exists $u\in\D H^{\ast}(t,y,\cdot)$ satisfying inequalities $|u-v|\leq k_R(t)|y-x|$ and $H^{\ast}(t,y,u)\leq H^{\ast}(t,x,v)+k_R(t)|y-x|$.

\tn{\bf{(ELC)}} For any $R>0$ there exists a measurable map $k_R:[0,T]\to[0,+\infty)$ such that 
$E_{H^{\ast}}(t,x)\,\subset\, E_{H^{\ast}}(t,y)+k_R(t)\,|x-y|\,(\B\times[-1,1])$ for all $t\in[0,T]$, $x,y\in\B_R$.

Equivalences hold for the same map $k_R(\cdot)$.
\end{Th}

For nonempty subsets $K$, $D$ of $\R^m$, the extended Hausdorff distance between $K$ and $D$ is defined by the formula
\begin{equation*}
\mathscr{H}(K,D):= \max\big\{\,\sup_{\xi\in K}d(\xi,D),\;\sup_{\xi\in D}d(\xi,K)\,\big\}\in\R\cup\{+\infty\}.
\end{equation*}

By Theorem \ref{tw2_rlhmh} (ELC)  we obtain the following corollary.

\begin{Cor}\label{hlc-cor-ner}
Assume  \tn{(H1)-(H3)} and \tn{(HLC)}. Then the following inequality
\begin{equation*}
\mathscr{H}(E_{H^{\ast}}(t,x),E_{H^{\ast}}(t,y))\leq 2k_R(t)\,|x-y|
\end{equation*}
holds for any $t\in[0,T]$, $x,y\in \B_R$ and $R>0$.
\end{Cor}


\section{Graphical and epigraphical representations of the Hamiltonian}\label{dfrentger} 

\noindent In this section we show differences between graphical and epigraphical representations of the Hamiltonian whose conjugate is unbounded on the effective domain.

\vspace{2mm}
Let us define the Hamiltonian $\textit{I\!H}:\R\times\R\rightarrow\R$ by the formula
\begin{equation*}
\textit{I\!H}(x,p):=\left\{
\begin{array}{ccl}
(\sqrt{|xp|}-1)^2 & \tn{if} & |xp|\,> 1, \\[1mm]
0 & \tn{if} & |xp|\leq 1.
\end{array}
\right.
\end{equation*}
This Hamiltonian satisfies (H1)-(H4) and (HLC). Its conjugate  $\textit{I\!H}^{\ast}:\R\times\R\rightarrow\R\cup\{+\infty\}$ has the following form
\begin{equation*}
\textit{I\!H}^{\ast}(x,v)=\left\{
\begin{array}{ccl}
+\infty & \tn{if} & v\not\in(-|x|,|x|\,),\;x\not=0,\\[1mm]
\frac{\displaystyle |v|}{\displaystyle |x|-|v|} & \tn{if} & v\in(-|x|,|x|\,),\;x\not=0, \\[2mm]
0 & \tn{if} & v=0,\; x=0,\\[0mm]
+\infty & \tn{if} & v\not=0,\;x=0.
\end{array}
\right.
\end{equation*}
The set $\D\textit{I\!H}^{\ast}(x,\cdot)=(-|x|,|x|\,)$  is not closed and the function $v\rightarrow \textit{I\!H}^{\ast}(x,v)$ is not bounded on this set for every $x\in\R\setminus\{0\}$. Moreover, the function  $(x,v)\rightarrow \textit{I\!H}^{\ast}(x,v)$  is not continuous on the set $\D \textit{I\!H}^{\ast}$, because $\lim_{i\rightarrow\infty}\textit{I\!H}^{\ast}\left(2/i,1/i\right)=1\not= 0=\textit{I\!H}^{\ast}(0,0)$.

Since  $\textit{I\!H}$ satisfies (H1)-(H4) and (HLC) we can construct an epigraphical representation  $(\R\times\R,f,l)$ of $\textit{I\!H}$ like in Theorem \ref{th-rprez-glo12} below. However, the method of constructing a graphical representation given in  \cite{F-S,FR} cannot be applied to  $\textit{I\!H}$, since the parametrization theorem of set-valued maps involves closed-valued maps. However,  $x\to\D\textit{I\!H}^{\ast}(x,\cdot)$ is not a closed-valued map. Therefore we cannot utilize this approach to parametrize $x\to\D\textit{I\!H}^{\ast}(x,\cdot)$. Nevertheless,  to parametrize $x\to\D\textit{I\!H}^{\ast}(x,\cdot)$ we can use an epigraphical representation $(\R\times\R,f,l)$ of $\textit{I\!H}$ from Theorem \ref{th-rprez-glo12}. Then  $f(x,\R\times\R)=\D\textit{I\!H}^{\ast}(x,\cdot)$. Let
\begin{equation}\label{gerd1}
\textit{l\!l}(x,a_1,a_2)=\textit{I\!H}^{\ast}(x,f(x,a_1,a_2)).
\end{equation}
Of course, $(\R\times\R,f,\textit{l\!l})$ is a graphical representation of $\textit{I\!H}$. However, the function $\textit{l\!l}$ at the point $(0,0,r)$ is discontinuous for $r>0$. Indeed, let $a_1=r|x|/(1+r)$ and $a_2=r$ with $x\in\R$, $r>0$. We observe  that $(a_1,a_2)\in\E\textit{I\!H}^{\ast}(x,\cdot)$. By the extra-property (A3) of Theorem \ref{th-rprez-glo12},
\begin{equation}\label{gerd2}
f(x,a_1,a_2)=a_1\;\;\;\;\;\tn{and}\;\;\;\;\; l(x,a_1,a_2)=a_2.
\end{equation}
Let $a_{1i}=r|x_i|/(1+r)$ and $a_{2i}=r$ with $x_i=1/i$, $r>0$. Then $\lim_{i\to\infty}(x_i,a_{1i},a_{2i})=(0,0,r)$. In view of \eqref{gerd2}, we have $f(x_i,a_{1i},a_{2i})=a_{1i}$ and $f(0,0,r)=0$. By \eqref{gerd1} we get $\textit{l\!l}(x_i,a_{1i},a_{2i})=\textit{I\!H}^{\ast}(x,a_{1i})=r$ and $\textit{l\!l}(0,0,r)=\textit{I\!H}^{\ast}(0,0)=0$. Suppose that $\textit{l\!l}$ is continuous. Then we get $r=\lim_{i\to\infty}\textit{l\!l}(x_i,a_{1i},a_{2i})=\textit{l\!l}(0,0,r)=0$.  This contradicts the fact that $r>0$.

\vspace{2mm}
The lack of continuity of the function $\textit{l\!l}$ is not suprising, because
the function $\textit{l\!l}$ is a composition of discontinuous and continuous functions.
Such compositions usually are not continuous. However, it is not a rule. We observe that the function $\textit{f\!f}(x,a)=a|x|^2/(1+|a|\,|x|)$ is parametrization of $\D\textit{I\!H}^{\ast}(x,\cdot)$ such that $\textit{f\!f}(x,\R)=\D\textit{I\!H}^{\ast}(x,\cdot)$. Then, the function $\textit{l\!l}(x,a)=\textit{I\!H}^{\ast}(x,\textit{f\!f}(x,a))=|a|\,|x|$ is continuous. Therefore $(\R,\textit{f\!f},\textit{l\!l})$ is a continuous graphical representation of $\textit{I\!H}$. In general, it is difficult to indicate the class of Hamiltonians with  discontinuous conjugates for which continuous graphical representations with the unbounded control set exist.

\vspace{2mm}
The Hamiltonian $\textit{I\!H}$ does not have a graphical representation $(\mathds{A},\textit{f\!f},\textit{l\!l})$ such that
\begin{equation}\label{gerd3}
|\textit{l\!l}(x,a)-\textit{l\!l}(y,a)|\leq k_R\,|x-y|\;\;\; \tn{for all}\;\;\, x,y\in\B_R,\; a\in\mathds{A}. 
\end{equation}
In particular, the Hamiltonian $\textit{I\!H}$ does not have a graphical representation $(\mathds{A},\textit{f\!f},\textit{l\!l})$  which satisfies (A1) from  Theorem \ref{th-rprez-glo12}. Let us assume by contradiction that  the Hamiltonian $\textit{I\!H}$ has a graphical representation $(\mathds{A},\textit{f\!f},\textit{l\!l})$ satisfying \eqref{gerd3}. Let $x_i=R/i$ and $v_i=R/(2i)$ with $R>0$. Because of $v_i\in\D\textit{I\!H}^{\ast}(x_i,\cdot)=\textit{f\!f}(x_i,\mathds{A})$, there exists $a_i\in\mathds{A}$ such that $\textit{f\!f}(x_i,a_i)=v_i$. We observe that 
$\textit{l\!l}(x_i,a_i)=\textit{I\!H}^{\ast}(x_i,\textit{f\!f}(x_i,a_i))=\textit{I\!H}^{\ast}(x_i,v_i)=1$ and $\textit{l\!l}(0,a_i)=\textit{I\!H}^{\ast}(0,\textit{f\!f}(0,a_i))=\textit{I\!H}^{\ast}(0,0)=0$. In view of \eqref{gerd3} we have $1=|\textit{l\!l}(x_i,a_i)-\textit{l\!l}(0,a_i)|\leq k_R\,|x_i|=K_R\,R/i$. Passing to the limit as $i\to\infty$ we obtain $1\leq 0$, a contradiction.


\section{Main results}\label{section-3}
\noindent In this section we describe the main results of the paper that concern  faithful representations $(A,f,l)$ with the unbounded control set $A:=\R^{n}\times\R$. 

\pagebreak
\begin{Th}[\bf{Representation}]\label{th-rprez-glo12}
Assume \tn{(H1)-(H4)}, \tn{(HLC)}. Then there exist functions  $f:[0,T]\times\R^n\times\R^{n+1}\rightarrow\R^n$ and $l:[0,T]\times\R^n\times\R^{n+1}\rightarrow\R$, measurable in $t$ for all $(x,a)\in\R^n\times\R^{n+1}$ and continuous in $(x,a)$ for all $t\in[0,T]$, such that for every $t\in[0,T]$, $x,p\in\R^n$
\begin{equation*}
 H(t,x,p)=\sup_{\;\;a\,\in\,\R^{n+1}}\,\{\,\langle\, p,f(t,x,a)\,\rangle-l(t,x,a)\,\}
\end{equation*}
and $f(t,x,\R^{n+1})=\D H^{\ast}(t,x,\cdot)$. Moreover, we have the following.
\begin{enumerate}
\item[\tn{\bf{(A1)}}] For any $R>0$, $t\in[0,T]$, $x,y\in \B_R$, $a,b\in\R^{n+1}$\vspace{-1mm}
\begin{equation*}
\begin{array}{l}
|f(t,x,a)-f(t,y,b)|\leq 10\,(n+1)\,(\,k_R(t)\,|x-y|+|a-b|\,),\\[0.3mm]
|l(t,x,a)-l(t,y,b)|\leq 10\,(n+1)\,(\,k_R(t)\,|x-y|+|a-b|\,).
\end{array}
\end{equation*}
\item[\tn{\bf{(A2)}}] For any $t\in[0,T]$, $x\in\R^n$, $a\in\R^{n+1}$\vspace{-1mm}
\begin{equation*}
\hspace{-2.4mm}\begin{array}{l}
|f(t,x,a)|\leq c(t)(1+|x|)\;\;\it{and}\;\;-|H(t,x,0)|\leq l(t,x,a),\\[0.3mm]
l(t,x,a)\leq 2|H(t,x,0)|+2c(t)(1+|x|)+3|a|.
\end{array}
\end{equation*}
\item[\tn{\bf{(A3)}}]   $a=(f(t,x,a),l(t,x,a))$\, for all\,  $a\in\E H^{\ast}(t,x,\cdot)$, $t\in[0,T]$, $x\in\R^n$.\vspace{1mm}
\item[\tn{\bf{(A4)}}] Additionally, if $H$ is continuous, so are $f$ and $l$.
\end{enumerate}
\end{Th}

 Property (A1) means that $f,l$  are locally Lipschitz continuous in $x$ with Lipschitz constants dependent on time and globally Lipschitz continuous in $a$ with Lipschitz constant independent on time. Property (A2) implies that $f$ has sublinear growth in $x$ and $l$ has sublinear growth in $a$. Property (A3) is called the extra-property. Property (A4) means that $f,l$  are  continuous if only $H$ is continuous. The proof of Thm.  \ref{th-rprez-glo12} is given in Sect. \ref{thm-profsrst}.

\begin{Rem}\label{extraproperty}
We consider the representation $(\R^{n+1}\!\!\!,f,l)$ of $H$ defined as in  Theorem~\ref{th-rprez-glo12}. Then, in view of \cite[Lem. 4.1]{AM}, we have $\e(t,x,\R^{n+1})\subset\E H^{\ast}(t,x,\cdot)$, where $\e=(f,l)$. Moreover, by the extra-property we get that $\E H^{\ast}(t,x,\cdot)=\e(t,x,\E H^{\ast}(t,x,\cdot))$. Hence,
\begin{equation*}
\E H^{\ast}(t,x,\cdot)=\e(t,x,\E H^{\ast}(t,x,\cdot))\subset \e(t,x,\R^{n+1})\subset\E H^{\ast}(t,x,\cdot).
\end{equation*}
Therefore, the extra-property implies that  $(\R^{n+1}\!\!\!,f,l)$ is an epigraphical representation of $H$ because $\e(t,x,\R^{n+1})=\E H^{\ast}(t,x,\cdot)$.  It turns out, however, that the extra-property is a much stronger property  than the equality $\e(t,x,\R^{n+1})=\E H^{\ast}(t,x,\cdot)$. Indeed, we consider the  absolutely continuous function $(x,u)(\cdot)$ from $[0,T]$ into $\R^n\times\R$ such that
\begin{equation*}
(\dot{x},\dot{u})(t)\in\E H^{\ast}(t,x(t),\cdot)\;\;\tn{a.e.}\;\;t\in[0,T]. 
\end{equation*}
Then in view of  Filippov Theorem and the equality $\e(t,x,\R^{n+1})=\E H^{\ast}(t,x,\cdot)$  there exists a measurable control $\hat{a}(\cdot)$ defined on $[0,T]$ with values in $\R^{n+1}$ such that 
\begin{equation*}
(\dot{x},\dot{u})(t)=\e(t,x(t),\hat{a}(t))\;\; \tn{a.e.}\;\; t\in[0,T]. 
\end{equation*}
Obviously, the measurable control $\hat{a}(\cdot)$ may  be not integrable. Whereas the extra-property  with the control $\check{a}(\cdot):=(\dot{x}(\cdot),\dot{u}(\cdot))$ implies that
\begin{equation*}
(\dot{x},\dot{u})(t)=\e(t,x(t),\check{a}(t))\;\; \tn{a.e.}\;\; t\in[0,T]. 
\end{equation*}
Since $(x,u)(\cdot)$ is an absolutely continuous function,   $(\dot{x},\dot{u})(\cdot)$ is an integrable function. Therefore, the control $\check{a}(\cdot)$ is also integrable. Moreover, by Theorem~\ref{th-rprez-glo12} (A2) we have
\begin{equation}\label{ofl}
|l(t,x(t),a(t))|\leq 2\,\omega(t,x(t))+3\,|a(t)|\;\; \tn{for all}\;\; t\in[0,T],
\end{equation}
where $\omega(t,x):=|H(t,x,0)|+c(t)(1+|x|)$. We observe that if  $\omega(\cdot,x(\cdot))$ and $a(\cdot)$ are integrable functions, then the function $l(\cdot,x(\cdot),a(\cdot))$ is integrable. However, if the control $a(\cdot)$ is a measurable function, then the function $l(\cdot,x(\cdot),a(\cdot))$ may  be not integrable. 
\end{Rem}

\begin{Th}\label{thm-rep-stab2}
Let $H_i,H$, $i\in\N$, be continuous and satisfy \tn{(H1)-(H4)}, \tn{(HLC)}.  We consider the representations $(\R^{n+1}\!\!\!,f_i,l_i)$ and $(\R^{n+1}\!\!\!,f,l)$ of $H_i$ and $H$, respectively,  defined as in the proof of Theorem~\ref{th-rprez-glo12}. If $H_i$ converge uniformly on compacts to $H$, then $f_i$ converge to $f$ and $l_i$ converge to $l$ uniformly on compacts in $[0,T]\times\R^n\times\R^{n+1}$.
\end{Th}

\begin{Th}\label{thm-rep-stab4}
Let $H_i,H$, $i\in\N$, satisfy \tn{(H1)-(H4)}, \tn{(HLC)}.  We consider the representations $(\R^{n+1}\!\!\!,f_i,l_i)$ and $(\R^{n+1}\!\!\!,f,l)$ of $H_i$ and $H$, respectively, defined as in the proof of Theorem~\ref{th-rprez-glo12}. If $H_i(t,\cdot,\cdot)$ converge uniformly on compacts to $H(t,\cdot,\cdot)$ for every $t\in[0,T]$, then $f_i(t,\cdot,\cdot)$ converge to $f(t,\cdot,\cdot)$ and $l_i(t,\cdot,\cdot)$ converge to $l(t,\cdot,\cdot)$ uniformly on compacts in $\R^n\times\R^{n+1}$ for every $t\in[0,T]$.
\end{Th}

\subsection{Correlation between variational and optimal control problems}\label{rvptoocp} 
In this subsection we consider a special kind of variational and optimal control problems  describing solutions of the  Hamilton-Jacobi equation with Hamiltonian which satisfies (H1)-(H4) and (HLC). 
 These problems are theoretical in nature. Nevertheless, they can be useful in investigating practical problems. For instance,  using these  variational and optimal control problems we prove stability of value functions and local Lipschitz continuity of the value function. We consider the following variational problem
\begin{equation}\label{problemcwp}
\begin{aligned}
\mathrm{minimize}&\;\;\;\Gamma[x(\cdot)]:=\phi(x(t_0),x(T))+\int_{t_0}^TH^{\ast}(t,x(t),\dot{x}(t))\,dt,\\[-1mm]
\mathrm{subject\;\, to}&\;\;\;x(\cdot)\in \mathcal{A}([t_0,T],\R^n),
\end{aligned}\tag{$\mathcal{P}_{v}$}
\end{equation}
and the following  optimal control problem 
\begin{equation}\label{problemcop}
\begin{aligned}
\mathrm{minimize}&\;\;\;\Lambda[(x,a)(\cdot)]:=\phi(x(t_0),x(T))+\int_{t_0}^Tl(t,x(t),a(t))\,dt,\\[-0.5mm]
\mathrm{subject\;\, to}&\;\;\;\dot{x}(t)=f(t,x(t),a(t))\;\;\mathrm{a.e.}
\;\,t\in[t_0,T],\\
\mathrm{and}&\;\;\; x(\cdot)\in \mathcal{A}([t_0,T],\R^n),\;a(\cdot)\in L^1([t_0,T],\R^{n+1}).
\end{aligned}\tag{$\mathcal{P}_{c}$}
\end{equation}

\begin{Th}\label{thm-reduct}
Assume that \tn{(H1)-(H4)} and \tn{(HLC)} hold with integrable functions $c(\cdot)$, $k_R(\cdot)$, $H(\cdot,0,0)$.  We consider the representation $(\R^{n+1}\!\!\!,f,l)$ of $H$ defined as in Theorem~\ref{th-rprez-glo12}. Assume further that $\phi$ is a proper, lower semicontinuous function  and  there exists $M\geq 0$ such that $\min\{\,|z|,|x|\,\}\leq M$ for all $(z,x)\in\D\phi$. Then
\begin{equation}
\min\Gamma[x(\cdot)]\,=\,\min\Lambda[(x,a)(\cdot)].
\end{equation}
Besides, if $\kr{x}(\cdot)$ is the optimal arc of \eqref{problemcwp} such that $\kr{x}(\cdot)\in\D\Gamma$, then $(\kr{x},\kr{a})(\cdot)$ is the optimal arc of \eqref{problemcop} with $\kr{a}(\cdot)=(\dot{\kr{x}}(\cdot),H^{\ast}(\cdot,\kr{x}(\cdot),\dot{\kr{x}}(\cdot)))$ such that $(\kr{x},\kr{a})(\cdot)\in\D\Lambda$.  Conversely, if $(\kr{x},\kr{a})(\cdot)$ is the optimal arc of \eqref{problemcop}, then $\kr{x}(\cdot)$ is the optimal arc of \eqref{problemcwp}.
\end{Th}

The indicator function $\psi_K(\cdot)$  of the set $K$ has value $0$ on this set and $+\infty$ outside. \linebreak Applying Theorem \ref{thm-reduct} to $\phi(z,x):=\psi_{\{x_0\}}(z)+g(x)$, we obtain the following corollary.

\pagebreak
\begin{Cor}\label{cor-reduct}
Assume that \tn{(H1)-(H4)} and \tn{(HLC)} hold with integrable functions $c(\cdot)$, $k_R(\cdot)$, $H(\cdot,0,0)$.  We consider the representation $(\R^{n+1}\!\!\!,f,l)$ of $H$ defined as in  Theorem~\ref{th-rprez-glo12}. Let $g$ be a proper, lower semicontinuous function. Then for all $(t_0,x_0)\in[0,T]\times\R^n$ 
\begin{eqnarray*}
V(t_0,x_0) &=& \min_{\begin{array}{c}
\scriptstyle x(\cdot)\,\in\,\mathcal{A}([t_0,T],\R^n)\\[-1mm]
\scriptstyle x(t_0)=x_0
\end{array}}\!\!\big\{\,g(x(T))+\int_{t_0}^TH^{\ast}(t,x(t),\dot{x}(t))\,dt\,\big\}\\
&=& \min_{(x,a)(\cdot)\,\in\, \emph{S}_f(t_0,x_0)}\,\big\{\,g(x(T))+\int_{t_0}^Tl(t,x(t),a(t))\,dt\,\big\}.
\end{eqnarray*}
\end{Cor}

\begin{Rem}\label{cor-bvocp-rem}Observe that the considered optimal control problem \eqref{problemcop} has integrable controls. Investigating integrable controls is possible due to argumentation contained 
in Remark \ref{extraproperty}; see Subsect. \ref{thm-reduct-sect}. In addition, correlation between the optimal control $\kr{a}(\cdot)$ and the optimal trajectory  $\kr{x}(\cdot)$ 
can be expressed by the simple formula  $(\dot{\kr{x}}(\cdot),H^{\ast}(\cdot,\kr{x}(\cdot),\dot{\kr{x}}(\cdot)))=\kr{a}(\cdot)$. Of course, this formula does not make sense in the case of the optimal control problem with the compact control set considered in paper \cite{AM}. Indeed, on the one hand $\dot{\kr{x}}(\cdot)$ and $H^{\ast}(\cdot,\kr{x}(\cdot),\dot{\kr{x}}(\cdot)))$ need not be be bounded functions. On the other hand, the control $\kr{a}(\cdot)$ has values in a compact control set.
\end{Rem}

\subsection{Stability of value functions}

\begin{Th}\label{cor-rep-stab2}
Let $H_i,H$, $i\in\N$, satisfy \tn{(H1)-(H4)} and \tn{(HLC)}  with the same  integrable functions $c(\cdot)$, $k_R(\cdot)$. We assume that $g_i$, $g$, $i\in\N$, are continuous functions  and $g_i$  converge to $g$ uniformly on compacts in $\R^n$.  Let $\max\{\,|H_i(t,0,0)|,|H(t,0,0)|\,\}\leq\mu(t)$ for all $t\in[0,T]$, $i\in\N$ and some integrable function $\mu(\cdot)$.   We consider the representations $(\R^{n+1}\!\!\!,f_i,l_i)$ and $(\R^{n+1}\!\!\!,f,l)$ of $H_i$ and $H$, respectively, defined as in the proof of Theorem~\ref{th-rprez-glo12}. If $V_i$ and $V$ are the value functions associated with $(\R^{n+1}\!\!\!,f_i,l_i,g_i)$ and $(\R^{n+1}\!\!\!,f,l,g)$, respectively, and $H_i(t,\cdot,\cdot)$ converge uniformly on compacts to $H(t,\cdot,\cdot)$  for all $t\in[0,T]$, then $V_i$ converge uniformly on compacts to $V$ in $[0,T]\times\R^n$.
\end{Th}

\begin{Def}
A sequence of functions $\{\varphi_i\}_{i\in\N}$, is said to \it{epi-converge} to  function $\varphi$ (e-$\lim_{i\to\infty}\varphi_i=\varphi$ for short) if, for every point $z\in\R^m$,
\begin{enumerate}
\item[\bf{(i)}] $\liminf_{i\to\infty}\varphi_i(z_i)\geq\varphi(z)$ for every sequence $z_i\to z$,
\item[\bf{(ii)}] $\limsup_{i\to\infty}\varphi_i(z_i)\leq\varphi(z)$ for some sequence $z_i\to z$.
\end{enumerate}
\end{Def}

\begin{Th}\label{cor-rep-stab4}
Let $H_i,H$, $i\in\N$, satisfy \tn{(H1)-(H4)} and \tn{(HLC)}  with the same  integrable functions $c(\cdot)$, $k_R(\cdot)$. Let $g_i$, $g$, $i\in\N$, be  proper, lower semicontinuous  and \tn{e-$\lim_{i\to\infty}g_i=g$}.\linebreak  Let $\max\{\,|H_i(t,0,0)|,|H(t,0,0)|\,\}\leq\mu(t)$ for all $t\in[0,T]$, $i\in\N$ and some integrable function $\mu(\cdot)$.   We consider the representations $(\R^{n+1}\!\!\!,f_i,l_i)$ and $(\R^{n+1}\!\!\!,f,l)$ of $H_i$ and $H$, respectively, defined as in the proof of Theorem~\ref{th-rprez-glo12}. If $V_i$ and $V$ are the value functions associated with $(\R^{n+1}\!\!\!,f_i,l_i,g_i)$ and $(\R^{n+1}\!\!\!,f,l,g)$, respectively, and $H_i(t,\cdot,\cdot)$ converge uniformly on compacts to $H(t,\cdot,\cdot)$  for all $t\in[0,T]$, then \tn{e-$\lim_{i\to\infty}V_i=V$}.
\end{Th}

\begin{Rem}
Proofs of stability of value functions in paper \cite{AM} were omitted, because they relate to simple methods. However, standard tools cannot be applied to Theorem~\ref{cor-rep-stab2} and Theorem \ref{cor-rep-stab4}, because we consider the optimal arc $(x_i,a_i)(\cdot)$ of $V_i(x_{i0},a_{i0})$ for all $i\in\N$. Fix  $t\in[0,T]$. Then one can prove that the sequence  $\{x_i(t)\}$ is bounded in  $\R^n$. However, the sequence $\{a_i(t)\}$ is, in general, not bounded in $\R^{n+1}$.  This means that to the sequence $\{(x_i(t),a_i(t))\}$ one cannot apply Theorem \ref{thm-rep-stab4}, because this theorem works only on compact  subsets of the set  $\R^n\times\R^{n+1}$. Therefore, we decided to strengthen\linebreak  Theorem \ref{thm-rep-stab4} to work on sets of the type $\B_R\times\R^{n+1}$. It can be done by assuming significantly stronger convergence of Hamiltonians than the one considered in this paper; see [arXiv:1507.01424v1, Theorem 3.14]. However, the strengthened Theorem \ref{thm-rep-stab4} turned out to be needless, because introducing the nonstandard method of the proof overcame the above problem; see Section~\ref{thm-reguvalufun}.
\end{Rem}

\subsection{Lipschitz continuous/continuous/lower semicontinuous of the value function}

\begin{Th}\label{lip-value-function-2}
Assume that \tn{(H1)-(H4)} and \tn{(HLC)} hold with integrable functions $c(\cdot)$, $k_R(\cdot)$, $H(\cdot,0,0)$.  We consider the representation $(\R^{n+1}\!\!\!,f,l)$ of $H$ defined as in  Theorem~\ref{th-rprez-glo12}. Let $g$ be a locally Lipschitz function. Assume that $V$ is the value function associated with  $(\R^{n+1}\!\!\!,f,l,g)$. Then for every $M>0$ there exist $\alpha_M(\cdot)\in\mathcal{A}([0,T],\R)$ and $C_M>0$ such that 
\begin{equation}\label{nlfw-12}
|V(t,x)-V(s,y)|\leq |\alpha_M(t)-\alpha_M(s)|+C_M|x-y|,\;\;\forall\,t,s\in[0,T],\;\forall\,x,y\in\B_M.
\end{equation}
Additionally, the value function $V(\cdot,\cdot)$ is locally Lipschitz continuous with respect to all variables on the set $[0,T]\times\R^n$, if  $c(\cdot)$, $k_R(\cdot)$ and $H(\cdot,\cdot,\cdot)$ are continuous functions.
\end{Th}

\begin{Rem}
We consider the optimal arc $(x_\pi,a_\pi)(\cdot)$ of $V(x_{\pi},a_{\pi})$ for all $\pi\in\Pi$. In the proof of Theorem \ref{lip-value-function-2} one requires equi-boundedness of the family $\{\,l(\cdot,x_\pi(\cdot),a_\pi(\cdot))\mid\pi\in\Pi\,\}$ by an integrable function. In view of \eqref{ofl} we obtain
\begin{equation*}
|\,l(\cdot,x_\pi(\cdot),a_\pi(\cdot))\,|\,\leq\, 2\,\omega(\cdot,x_\pi(\cdot))+3\,|\,a_\pi(\cdot)\,|\;\;\; \tn{for all}\;\; \;\pi\in\Pi.
\end{equation*}
One can prove that the family  $\{\,x_\pi(\cdot)\mid\pi\in\Pi\,\}$ is equi-bounded by a constant function. However, the family $\{\,a_\pi(\cdot)\mid\pi\in\Pi\,\}$ is not bounded in general. Whereas, if $g$ is a locally\linebreak Lipschitz function, then there exists an integrable function
that equi-bounds the family $\{\,a_\pi(\cdot)\mid\pi\in\Pi\,\}$; see [arXiv:1807.03640v1, Theorem 4.7]. So, the family  $\{\,l(\cdot,x_\pi(\cdot),a_\pi(\cdot))\mid\pi\in\Pi\,\}$ can be bounded by an integrable function. It turns out that proceeding adequately in the proof of 
Theorem \ref{lip-value-function-2} we can omit equi-boundedness of optimal controls; see Section \ref{thm-reguvalufun}. In the literature the above problem is solved assuming 
boundedness of the function $l$ independent of $a$; see \cite{B-CD}.
\end{Rem}

\begin{Th}\label{con-value-function}
Assume that \tn{(H1)-(H4)} and \tn{(HLC)} hold with integrable functions $c(\cdot)$, $k_R(\cdot)$, $H(\cdot,0,0)$.  We consider the representation $(\R^{n+1}\!\!\!,f,l)$ of $H$ defined as in  Theorem~\ref{th-rprez-glo12}. Let $g$ be a continuous/lower semicontinuous function. Then the value function associated with  $(\R^{n+1}\!\!\!,f,l,g)$ is continuous/lower semicontinuous on $[0,T]\times\R^n$.
\end{Th}

\begin{Rem}
Theorem \ref{con-value-function} is a direct consequence of the proofs of Theorems \ref{cor-rep-stab2} and \ref{cor-rep-stab4}; see  Remarks \ref{rem-uscvvfs} and \ref{rem-lscvvfs}.
\end{Rem}


\section{Representation, optimality and stability theorems}\label{thm-profsrst}

\noindent The support function $\sigma(K,\cdot):\R^m\to\R$ of a nonempty, convex, compact set $K\subset\R^m$ is a convex  real-valued function defined by
\begin{equation*}
\sigma(K,p):=\max_{x\in K}\,\langle p,x\rangle,\quad \forall\,p\in\R^m.
\end{equation*}

Let $\sum_{^{\, m-1}}$ denotes the unit sphere in $\R^m$ and let $\mu$ be the measure on $\sum_{^{\, m-1}}$ proportional to the Lebesgue measure and satisfying $\mu(\sum_{^{\, m-1}})=1$.

\begin{Def}\label{df-scel}
Let $m\in\N\setminus\{1\}$. For any nonempty, convex, compact subset $K$ of $\R^m$, its Steiner point is defined by
\begin{equation*}
s_m(K):=m\int_{\sum_{^{\, m-1}}}p\,\sigma(K,p)\;\mu(dp).
\end{equation*}
One can show that $s_m(\cdot)$ is a selection in the sense that $s_m(K)\in K$, cf. \cite[p. 366]{A-F}.
\end{Def}

\begin{Th}\label{th-oparam}
Let a set-valued map $E:[0,T]\times\R^n\multimap\R^m$ satisfy \tn{(E1)-(E4)}. Then there exists a single-valued map $\e\!:\![0,\!T]\!\times\!\R^n\!\times\!\R^m\!\to\!\R^m$ such that $\e(\cdot,x,a)$ is measurable for~all  $(x,a)\!\in\!\R^n\!\times\!\R^m$ and $\e(t,\cdot,\cdot)$ is continuous  for all $t\!\in\![0,\!T]$. Moreover, we have the following.
\begin{enumerate}
\item[$\pmb{(\tn{a}_1)}$] $\e(t,x,\R^m)=E(t,x)$ \,for all\,  $t\in[0,T]$, $x\in\R^n$\tn{;} \vspace{1mm}
\item[$\pmb{(\tn{a}_2)}$] $a=\e(t,x,a)$ \,for all\,  $a\in E(t,x)$, $t\in[0,T]$, $x\in\R^n$\tn{;} \vspace{1mm}
\item[$\pmb{(\tn{a}_3)}$] $|\e(t,x,a)|\leq 3|a|+2d(0,E(t,x))$ \,for all\,  $a\in \R^m$, $t\in[0,T]$, $x\in\R^n$\tn{;} \vspace{1mm}
\item[$\pmb{(\tn{a}_4)}$] $|\e(t,x,a)-\e(t,y,b)|\;\leq\; 5m\,[\,\mathscr{H}(E(t,x),E(t,y))+|a-b|\,]$  \;\;for all\\ $t\in[0,T]$,\, $x,y\in\R^n$ \,and\; $a,b\in\R^m$\tn{;}\vspace{1mm}
\item[$\pmb{(\tn{a}_5)}$] Additionally, if \tn{(E6)}  is verified, then $\e(\cdot,\cdot,\cdot)$ is continuous.
\end{enumerate}
\end{Th}

\begin{proof}
Let $(t,x,a)\in[0,T]\times\R^n\times\R^m$. We consider the set-valued map defined by
$$\Phi(t,x,a):= E(t,x)\cap \B(a,2d(a,E(t,x))).$$
We observe that the set-valued map $\Phi$ is defined as in the proof of \cite[Theorem 5.6]{AM}, if we assume that $\omega\equiv 1$. We define the single-valued map $\e$ from $[0,T]\times\R^n\times\R^m$ to $\R^m$ by
$$\e(t,x,a):= s_m(\Phi(t,x,a)),$$
where $s_m$ in the Steiner selection. Since $\Phi$ is defined as in the proof of \cite[Theorem 5.6]{AM}, so the single-valued map $\e$  is well-defined. Moreover, $\e(\cdot,x,a)$ is measurable for every  $(x,a)\in\R^n\times\R^m$ and $\e(t,\cdot,\cdot)$ is continuous  for every $t\in[0,T]$. If  we assume that $\omega\equiv 1$, then by \cite[Theorem 5.6]{AM} and \cite[Lemma 5.1]{AM}  we obtain ($\tn{a}_4$) and ($\tn{a}_5$). It remains to prove ($\tn{a}_1$)-($\tn{a}_3$).

 Let us notice that by Definition~\ref{df-scel} we obtain for all $t\in[0,T]$, $x\in\R^n$, $a\in\R^m$,
\begin{equation}\label{parw1}
\e(t,x,a)=s_m(\Phi(t,x,a))\in \Phi(t,x,a)=E(t,x)\cap \B(a,2d(a,E(t,x))).
\end{equation}

To prove ($\tn{a}_2$) we observe that by \eqref{parw1} we get $|\e(t,x,a)-a|\leq 2d(a,E(t,x))$ for every $t\in[0,T]$, $x\in\R^n$, $a\in\R^m$. Hence $a=\e(t,x,a)$ for all $a\in E(t,x)$, $t\in[0,T]$, $x\in\R^n$.

To prove ($\tn{a}_1$) we observe that by \eqref{parw1} we get $\e(t,x,\R^m)\!\subset\! E(t,x)$ for all $(t,\!x)\!\in\![0,\!T]\!\times\!\R^n$. The latter, together with ($\tn{a}_2$), implies that 
$E(t,x)\,=\, \e(t,x,E(t,x))\,\subset\, \e(t,x,\R^m)\,\subset\, E(t,x)$ for every $(t,x)\in[0,T]\times\R^n$. This means that $\e(t,x,\R^m)=E(t,x)$ for all $(t,x)\in[0,T]\times\R^n$.

To prove ($\tn{a}_3$) we observe that by \eqref{parw1} we get $|\e(t,x,a)\!-\!a|\!\leq\! 2d(a,E(t,x))$ for all $t\!\in\![0,\!T]$, $x\in\R^n$, $a\in\R^m$. The latter, together with the inequality $d(a,E(t,x))\leq d(0,E(t,x))+|a|$, implies that $|\e(t,x,a)|\leq 3|a|+2d(0,E(t,x))$ for all  $a\in \R^m$, $t\in[0,T]$, $x\in\R^n$.
\end{proof}

\begin{Th}\label{do-th-parame-lo1}
Assume that  $H$ satisfies  \tn{(H1)-(H4)}, \tn{(HLC)}.  Then there exists a function $\e:[0,T]\times\R^n\times\R^{n+1}\to\R^{n+1}$ such that $\e(\cdot,x,a)$ is measurable for every  $(x,a)\in\R^n\times\R^{n+1}$ and $\e(t,\cdot,\cdot)$ is continuous  for every $t\in[0,T]$. Moreover, we have the following.
\begin{enumerate}
\item[\tn{$\pmb{(\e_1)}$}] $\e(t,x,\R^{n+1})=\E H^{\ast}(t,x,\cdot)$ \, for all\,  $t\in[0,T]$, $x\in\R^n$\tn{;} \vspace{1mm}
\item[\tn{$\pmb{(\e_2)}$}] $a=\e(t,x,a)$\, for all\,  $a\in \E H^{\ast}(t,x,\cdot)$, $t\in[0,T]$, $x\in\R^n$\tn{;} \vspace{1mm}
\item[\tn{$\pmb{(\e_3)}$}] $|\e(t,x,a)|\leq 2|H(t,x,0)|+2c(t)(1+|x|)+3|a|$\, for all\,   $t\in[0,T]$, $x\in\R^n$, $a\in \R^{n+1}$\tn{;} \vspace{1mm}
\item[\tn{$\pmb{(\e_4)}$}] $|\e(t,x,a)-\e(t,y,b)|\leq 10\,(n+1)\,(\,k_R(t)\,|x-y|+|a-b|\,)$  \,for all\, $t\in[0,T]$, $x,y\in \B_R$, $a,b\in\R^{n+1}$\;and\; $R>0$\tn{;}\vspace{1mm}
\item[\tn{$\pmb{(\e_5)}$}] Additionally, if $H(\cdot,\cdot,\cdot)$ is continuous, so is $\e(\cdot,\cdot,\cdot)$.
\end{enumerate}
\end{Th}

\begin{proof}
Let $E(t,x):=E_{H^{\ast}}(t,x)=\E H^{\ast}(t,x,\cdot)$ for every $(t,x)\in[0,T]\times\R^n$. Because of  Corollaries \ref{wrow-wm} and \ref{hlc-cor-ner},  the function $E$ satisfies assumptions of Theorem \ref{th-oparam}. Therefore, there exists a  map $\e:[0,T]\times\R^{n}\times\R^{n+1}\to\R^{n+1}$ such that $\e(\cdot,x,a)$ is measurable for every  $(x,a)\in\R^n\times\R^{n+1}$ and $\e(t,\cdot,\cdot)$ is continuous  for every $t\in[0,T]$. Moreover, the map $\e(\cdot,\cdot,\cdot)$ satisfies ($\tn{a}_1$)-($\tn{a}_5$) from Theorem \ref{th-oparam}. By Theorem \ref{th-oparam} ($\tn{a}_4$) and Corollary~\ref{hlc-cor-ner} we have
\begin{eqnarray*}
|\e(t,x,a)-\e(t,y,b)| &\leq & 5(n+1)[\,\mathscr{H}(E(t,x),E(t,y))+|a-b|\,]\\
&\leq & 10(n+1)\,k_R(t)\,|x-y|+5(n+1)|a-b|
\end{eqnarray*}
for all $t\in[0,T]$, $x,y\in \B_R$, $a,b\in \R^{n+1}$, $R>0$. It means that the inequality  ($\e_4$) is satisfied. Moreover, if we assume that $H$ is continuous, then by Corollary~\ref{wrow-wm} we get that (E6)  is verified. Thus, because of Theorem \ref{th-oparam} ($\tn{a}_5$), we obtain that the map $\e(\cdot,\cdot,\cdot)$ is continuous. We observe that ($\e_1$) and ($\e_2$) follows from definition of the set-valued map $E(\cdot,\cdot)$ and the properties  ($\tn{a}_1$) and ($\tn{a}_2$) in Theorem \ref{th-oparam}. It remains to prove ($\e_3$). 

Fix $(t,x)\in[0,T]\times\R^n$. Because of  (C1)-(C5), there exists $\kr{v}\in \D H^{\ast}(t,x,\cdot)$ such that $H(t,x,0)=H^{\ast\ast}(t,x,0)=-H^{\ast}(t,x,\kr{v})$ and $|\kr{v}|\leq c(t)(1+|x|)$. We see $(\kr{v},H^{\ast}(t,x,\kr{v}))\in E(t,x)$. The latter, together with  Theorem \ref{th-oparam} ($\tn{a}_3$), implies that
\begin{eqnarray*}
|\e(t,x,a)| &\leq & 3|a|+2d(0,E(t,x))\;\;\leq\;\;3|a|+2|(\kr{v},H^{\ast}(t,x,\kr{v}))|\;\;\leq\;\;\\
&\leq & 3|a|+2|\kr{v}|+2|H^{\ast}(t,x,\kr{v})|\;\;\leq\;\;3|a|+2c(t)(1+|x|)+2|H(t,x,0)|.
\end{eqnarray*}
This completes the proof of the theorem.
\end{proof}

\begin{Prop}[\tn{\cite[Prop. 5.7]{AM}}]\label{prop-reprezentacja H-ogr}
Let  $\e(t,x,\cdot)$ be a function defined on $\R^{n+1}$ into $\R^{n+1}$.  Assume that $H(t,x,\cdot)$ is a real-valued convex function and $\e(t,x,\R^{n+1})=\E H^{\ast}(t,x,\cdot)$.\linebreak If $\e(t,x,a)=(f(t,x,a),l(t,x,a))$ for all $a\in \R^{n+1}\!\!$, then the triple $(A,f,l)$ is a representation of $H$. Moreover, $f(t,x,\R^{n+1})=\D H^{\ast}(t,x,\cdot)$.
\end{Prop}

\begin{Rem}\label{dct-rem}
Let $\e:[0,T]\times\R^n\times\R^{n+1}\rightarrow\R^{n+1}$ be the function from Theorem \ref{do-th-parame-lo1}.\linebreak We define two functions $f:[0,T]\times\R^n\times\R^{n+1}\rightarrow\R^n$ and $l:[0,T]\times\R^n\times\R^{n+1}\rightarrow\R$ by 
\begin{equation*}
f(t,x,a):=\pi_1(\e(t,x,a))\;\;\;\tn{and}\;\;\;l(t,x,a):=\pi_2(\e(t,x,a)),
\end{equation*}
where  $\pi_1(v,\eta)=v$ and $\pi_2(v,\eta)=\eta$ for all $v\in\R^n$ and $\eta\in\R$. Then for all $t\in[0,T]$, $x\in\R^n$, $a\in\R^{n+1}$ the equality $\e(t,x,a)=(f(t,x,a),l(t,x,a))$ holds.
Therefore, for all $t\in[0,T]$, $x,y\in\R^n$, $a,b\in\R^{n+1}$ we obtain $|\,l(t,x,a)\,|\leq |\,\e(t,x,a)\,|$ and
\begin{eqnarray*}
|f(t,x,a)-f(t,y,b)| &\!\!\leq\!\! &  |\e(t,x,a)-\e(t,y,b)|,\\
|\:l(t,x,a)\:-\:l(t,y,b)\:| &\!\!\leq\!\! & |\e(t,x,a)-\e(t,y,b)|.\\
\end{eqnarray*}
\vspace*{-1.4cm}

\pagebreak
\noindent From the above inequalities it follows that  the  properties of the function $\e$ are inherited by functions $f,l$. It is not difficult to show that Theorem \ref{th-rprez-glo12} follows from Proposition \ref{prop-reprezentacja H-ogr}, Theorem \ref{do-th-parame-lo1} and Corollary \ref{wrow-wm} (E5).
\end{Rem}

\subsection{The optimality theorem}\label{thm-reduct-sect}
The proof of Theorem \ref{thm-reduct} is similar to the proof of \linebreak  \cite[Theorem 3.13]{AM}, so we omit it.  In this subsection we describe only the differences in these proofs due to the extra-property. Let $\mathcal{I}_{f}([t_0,T],\R^{2n+1})$ [resp. $\mathcal{M}_{f}([t_0,T],\R^{2n+1})$] denotes the set of all absolutely-integrable [resp. absolutely-measurable] pairs $(x,a)(\cdot)$ which satisfies $\dot{x}(t)=f(t,x(t),a(t))$ for a.e. $t\in[t_0,T]$. Analogously, as in \cite[Section 7]{AM}, we can show that  
 the functionals $\Gamma[\cdot]$ and $\Lambda[\cdot]$ are well-defined and
\begin{equation}\label{roz6-l2}
-\infty\,<\,\inf_{x(\cdot)\,\in\,\mathcal{A}([t_0,T],\R^n)}\Gamma[x(\cdot)],\quad
-\infty\,<\,\inf_{(x,\,a)(\cdot)\,\in\, \mathcal{M}_{f}([t_0,T],\R^{2n+1})}\Lambda[(x,a)(\cdot)].
\end{equation}

The differences in the proofs of Theorem \ref{thm-reduct} and \cite[Theorem 3.13]{AM}  are related to the following equalities:
\begin{eqnarray}
\inf_{x(\cdot)\,\in\,\mathcal{A}([t_0,T],\R^n)}\Gamma[x(\cdot)] &=& \inf_{(x,\,a)(\cdot)\,\in\, \mathcal{M}_{f}([t_0,T],\R^{2n+1})}\Lambda[(x,a)(\cdot)],\label{roz6-rthm1}\\[3mm]
\inf_{x(\cdot)\,\in\,\mathcal{A}([t_0,T],\R^n)}\Gamma[x(\cdot)] &=& \inf_{(x,\,a)(\cdot)\,\in\, \mathcal{I}_{f}([t_0,T],\R^{2n+1})}\Lambda[(x,a)(\cdot)].\label{roz6-rthm01}
\end{eqnarray}
Using  \cite[Lemma 4.1]{AM} we can show that  $\tn{LS}(\ref{roz6-rthm1})\leq\tn{RS}(\ref{roz6-rthm1})$; see \cite[Section 7]{AM}. The latter, together with $\mathcal{I}_{f}([t_0,T],\R^{2n+\!1})\!\subset\!\mathcal{M}_{f}([t_0,T],\R^{2n+\!1})$, implies that $\tn{LS}(\ref{roz6-rthm01})\!\leq\!\tn{RS}(\ref{roz6-rthm01})$.
The proofs of the opposite inequalities require an appropriate definition of control; see \cite[Section 7]{AM}. Using Filippov Theorem to define a measurable control; see Remark~\ref{extraproperty}, we can show that $\tn{LS}(\ref{roz6-rthm1})\geq\tn{RS}(\ref{roz6-rthm1})$. Whereas, using the extra-property to define an integrable control; see Remark \ref{extraproperty}, we can show that $\tn{LS}(\ref{roz6-rthm01})\geq\tn{RS}(\ref{roz6-rthm01})$. Therefore, without the extra-property we could show only the equality \eqref{roz6-rthm1}. Whereas, having the extra-property  we can prove both equalities  \eqref{roz6-rthm1} and \eqref{roz6-rthm01}.

From the equality \eqref{roz6-rthm01} and its proof it follows that if $\kr{x}(\cdot)$ is the optimal arc of \eqref{problemcwp} such that $\kr{x}(\cdot)\in\D\Gamma$, then $(\kr{x},\kr{a})(\cdot)$ is the optimal arc of \eqref{problemcop} with $\kr{a}(\cdot)=(\dot{\kr{x}}(\cdot),H^{\ast}(\cdot,\kr{x}(\cdot),\dot{\kr{x}}(\cdot)))$ such that $(\kr{x},\kr{a})(\cdot)\in\D\Lambda$;  conversely, if $(\kr{x},\kr{a})(\cdot)$ is the optimal arc of \eqref{problemcop}, then $\kr{x}(\cdot)$ is the optimal arc of \eqref{problemcwp}. Using the latter with  \cite[Theorem 7.6]{AM} we can replace ``$\inf$'' by ``$\min$'' in the equalities \eqref{roz6-rthm1} and \eqref{roz6-rthm01}.

\subsection{The stability theorems} The proofs of Theorems \ref{thm-rep-stab2} and \ref{thm-rep-stab4} are  consequences of  \cite[Theorem 6.6 and Remark 6.7]{AM}, if we assume that $\omega_i\equiv 1$ for all $i\in\N\cup\{0\}$ and $H_0=H$. In  \cite{AM} one assumed that for all $i\in\N\cup\{0\}$  the function $\omega_i$ is given by
\begin{equation*}
\omega_i(t,x)=|\lambda_i(t,x)|+|H_i(t,x,0)|+c_i(t)(1+|x|)+1\;\;\tn{with}\;\;\;t\in[0,T],\;x\in\R^n,
\end{equation*}
where  $c_i$ is coefficient in (H4) and $\lambda_i$ is upper boundedness of $H_i^{\ast}$. In \cite[Theorem 6.6]{AM} convergence  $\omega_i$ to $\omega_0$ is required. For this reason, in \cite[Theorems 3.8 and 3.9]{AM} one assumes convergence $H_i$ to $H_0$ as well as convergence  $\lambda_i$ to $\lambda_0$ and  $c_i$ to  $c_0$. Since, in our case $\omega_i\equiv 1$, so Theorems \ref{thm-rep-stab2} and \ref{thm-rep-stab4} 
do not need convergence $c_i$ to $c_0$.

\pagebreak


\section{Regularities of value functions}\label{thm-reguvalufun}

\noindent   Given real numbers $\tau$ and $\nu$, we put $\tau\wedge\nu:=\min\{\,\tau,\nu\,\}$  and $\tau\vee\nu:=\max\{\,\tau,\nu\,\}$. Let $\emph{S}_f(t_0,x_0)$ denotes the set of all trajectory-control pairs $(x,a)(\cdot)$ of the control system

\begin{equation*}
\begin{array}{ll}
\dot{x}(t)=f(t,x(t),a(t))&\mathrm{a.e.}\;\;t\in[0,T],\\[0mm]
x(t_0)=x_0,&
\end{array}
\end{equation*}
where $(x,a)(\cdot)\in \mathcal{A}([0,T],\R^n)\times L^1([0,T],\R^{n+1})$ and $(t_0,x_0)\in[0,T]\times\R^n$.

\subsection{Upper semicontinuity of value functions}\label{thm-reguvalufun-2}
Assume that $H_i,H$, $i\in\N$, satisfies  \tn{(H1)-(H4)} and \tn{(HLC)}  with the same  integrable maps $c(\cdot)$, $k_R(\cdot)$. We consider the representations $(\R^{n+1}\!\!\!,f_i,l_i)$ and $(\R^{n+1}\!\!\!,f,l)$ of $H_i$ and $H$, respectively, defined as in Theorem~\ref{th-rprez-glo12}.  Assume that there exists integrable function $\mu(\cdot)$ such that $|H_i(t,0,0)|\vee|H(t,0,0)|\leq\mu(t)$ for all $t\in[0,T]$, $i\in\N$. Moreover, we assume that   $H_i(t,\cdot,\cdot)$ converge uniformly on compacts to $H(t,\cdot,\cdot)$ for all $t\in[0,T]$. Then, in view of Theorem \ref{thm-rep-stab4}, $f_i(t,\cdot,\cdot)$ converge to $f(t,\cdot,\cdot)$ and $l_i(t,\cdot,\cdot)$ converge to $l(t,\cdot,\cdot)$ uniformly on compacts in $\R^n\times\R^{n+1}$ for all $t\in[0,T]$.

\begin{Th}\label{d-upperscvfs-1}
Let  $(\R^{n+1}\!\!\!,f_i,l_i)$ and $(\R^{n+1}\!\!\!,f,l)$ be as above. Assume that $g_i$ and $g$ are\linebreak continuous functions and $g_i$  converge to $g$ uniformly on compacts in $\R^n$. Let $V_i$ and $V$ be the value functions associated with $(\R^{n+1}\!\!\!,f_i,l_i,g_i)$ and $(\R^{n+1}\!\!\!,f,l,g)$, respectively.  Then for every $(t_0,x_0)\in[0,T]\times\R^n$ we have
\begin{equation*}
\limsup\nolimits_{\,i\,\to\,\infty}V_i(t_{i0},x_{i0})\leq V(t_0,x_0)\;\;\;\it{for every sequence}\;\;\; (t_{i0},x_{i0})\rightarrow (t_0,x_0).
\end{equation*}
\end{Th}

\begin{proof}
Let us fix $t_{i0},t_0\in[0,T]$, $x_{i0},x_0\in\R^n$ such that $(t_{i0},x_{i0})\rightarrow (t_0,x_0)$. Then there exists $M>0$ such that $x_{i0},x_0\in\B_M$. Without loss of generality we may assume $V(t_0,x_0)<+\infty$ since otherwise there is nothing to prove. Then by Corollary \ref{cor-reduct} there exists the optimal arc $(\kr{x},\kr{a})(\cdot)$ of $V(t_0,x_0)$ defined on $[t_0,T]$. We extend $\kr{a}(\cdot)$ from $[t_0,T]$ to $[0,T]$ by setting $\kr{a}(t)=0$ for $t\in[0,t_0]$. Next, because of the sublinear growth of $f$, we extend  $\kr{x}(\cdot)$ from $[t_0,T]$ to $[0,T]$ such that $(\kr{x},\kr{a})(\cdot)\in\emph{S}_f(t_0,x_0)$. Now we choose $x_i(\cdot)$ defined on $[0,T]$ such that $(x_i,\kr{a})(\cdot)\in\emph{S}_{f_i}(t_{i0},x_{i0})$. Then, our assumptions and Gronwall’s Lemma implies that
\begin{align}
&\|x_i\|\vee\|\kr{x}\|\;\leq\;\Big(M+\int\limits_{\scriptscriptstyle 0}^{\scriptscriptstyle T}c(t)\,dt\Big)\,\exp\Big(\int\limits_{\scriptscriptstyle 0}^{\scriptscriptstyle T}c(t)\,dt\Big)=:R, \label{rofvflem1}\\[-2mm]
&\|x_i-\kr{x}\|\;\leq\; \Big(|x_{i0}-x_0|+\|\,f_i[x_i]-f[x_i]\,\|_{L^1}+\!\!\int\limits_{\scriptscriptstyle t_{i0}\wedge t_0}^{\scriptscriptstyle t_{i0}\vee t_0}\omega_{\scriptscriptstyle R}[t]\,dt\Big)\,\exp\Big(\int\limits_{\scriptscriptstyle 0}^{\scriptscriptstyle T}\omega_{\scriptscriptstyle R}[t]\,dt\Big),\label{rofvflem2}
\end{align}
where $\,\omega_{\scriptscriptstyle R}[\cdot]:=2\,\mu(\cdot)+(10\,(n+1)\,k_R(\cdot)+2c(\cdot))(1+R)$, and $\,f_{i}[x_i](\cdot):=f_i(\cdot,x_i(\cdot),\kr{a}(\cdot))\,$,\linebreak $\,f[x_i](\cdot)\!:=\!f(\cdot,x_i(\cdot),\kr{a}(\cdot))\,$, and $\|\cdot\|_{L^1}$ is the standard norm in $L^1([0,\!T],\R^{m})$. We notice that
\begin{eqnarray}
\|\,f_i[x_i]-f[x_i]\,\|_{L^1} &=& \int_0^T|\,f_i(t,x_i(t),\kr{a}(t))-f(t,x_i(t),\kr{a}(t))\,|\, dt\nonumber\\
&\leq & \int_0^T\sup\nolimits_{z\in\B_R}|\,f_i(t,z,\kr{a}(t))-f(t,z,\kr{a}(t))\,|\, dt.\label{dn-upperscvfs-2}
\end{eqnarray}
Let $\Psi_i(\cdot):=\sup\nolimits_{z\in\B_R}|\,f_i(\cdot,z,\kr{a}(\cdot))-f(\cdot,z,\kr{a}(\cdot))\,|$. By Theorem~\ref{th-rprez-glo12} (A2) we get $2c(t)(1+R)\geq \Psi_i(t)$  for all $t\in[0,T]$. Since $f_i(t,\cdot,\kr{a}(t))$ converge to $f(t,\cdot,\kr{a}(t))$ uniformly on compacts in $\R^n$   for all $t\in[0,T]$, we have $\lim_{i\to\infty}\Psi_i(t)=0$ for all $t\in[0,T]$. Therefore by \eqref{dn-upperscvfs-2}
\begin{equation}\label{dn-upperscvfs-3}
\lim\nolimits_{i\to\infty}\|\,f_i[x_i]-f[x_i]\,\|_{L^1}=0.
\end{equation}
Observe that \eqref{rofvflem2}, together with \eqref{dn-upperscvfs-3}, implies $\lim_{i\to\infty}\|x_i-\kr{x}\|=0$. Since $l_i(t,\cdot,\kr{a}(t))$ and $l(t,\cdot,\kr{a}(t))$ are continuous, $l_i(t,\cdot,\kr{a}(t))$ converge to $l(t,\cdot,\kr{a}(t))$ uniformly on compacts in $\R^n$ and $x_i(t)\to \kr{x}(t)$ for all $t\in[0,T]$, we have $l_i(t,x_i(t),\kr{a}(t))\to l(t,\kr{x}(t),\kr{a}(t))$ for all $t\in[0,T]$. By Theorem~\ref{th-rprez-glo12} (A2) we get $|l_i(t,x_i(t),\kr{a}(t))|\leq\omega_R[t]+3|\kr{a}(t)|$ for all $t\in[0,T]$. Therefore
\begin{equation}\label{dn-upperscvfs-4}
\lim\nolimits_{i\to\infty}\|\,l_i[x_i]-l[\kr{x}]\,\|_{L^1}=0.
\end{equation}
Again by our assumptions and Gronwall’s Lemma  we obtain
\begin{eqnarray}
V(t_0,x_0) &=& g(\kr{x}(T))+\int\limits_{\scriptscriptstyle t_0}^{\scriptscriptstyle T}l[\kr{x}](t)\,dt\nonumber\\
&\geq & g(\kr{x}(T))+
\int\limits_{\scriptscriptstyle t_{i0}}^{\scriptscriptstyle T}l_{i}[x_i](t)\,dt-\|\,l_i[x_i]-l[\kr{x}]\,\|_{L^1}- \int\limits_{\scriptscriptstyle t_{i0}\wedge t_0}^{\scriptscriptstyle t_{i0}\vee t_0}\omega_{\scriptscriptstyle R}[t]\,dt-3\!\!\!\int\limits_{\scriptscriptstyle t_{i0}\wedge t_0}^{\scriptscriptstyle t_{i0}\vee t_0}|\,\kr{a}(t)\,|\,dt\nonumber\\
&\geq & g(\kr{x}(T))\!-\!g_i(x_i(T))+
V_i(t_{i0},x_{i0})-\|\,l_i[x_i]-l[\kr{x}]\,\|_{L^1}-\!\!\!\int\limits_{\scriptscriptstyle t_{i0}\wedge t_0}^{\scriptscriptstyle t_{i0}\vee t_0}\!\!\!\!\omega_{\scriptscriptstyle R}[t]\,dt-3\!\!\!\int\limits_{\scriptscriptstyle t_{i0}\wedge t_0}^{\scriptscriptstyle t_{i0}\vee t_0}\!\!|\,\kr{a}(t)\,|\,dt,\label{dn-upperscvfs-5}
\end{eqnarray}
where $\,l_{i}[x_i](\cdot):=l_i(\cdot,x_i(\cdot),\kr{a}(\cdot))\,$ and $\,l[\kr{x}](\cdot):=l(\cdot,\kr{x}(\cdot),\kr{a}(\cdot))$. Since $g_i$ and $g$ are continuous functions, $g_i$  converge to $g$ uniformly on compacts in $\R^n$, and $x_i(T)\to \kr{x}(T)$, we see that $g_i(x_i(T))\to g(\kr{x}(T))$. The latter, together with \eqref{dn-upperscvfs-4} and \eqref{dn-upperscvfs-5}, imply the following inequality $\limsup_{\,i\,\to\,\infty}V_i(t_{i0},x_{i0})\leq V(t_0,x_0)$.
\end{proof}

\begin{Th}\label{d-upperscvfs-2}
Let  $(\R^{n+1}\!\!\!,f_i,l_i)$ and $(\R^{n+1}\!\!\!,f,l)$ be as above. Assume that $g_i$ and $g$ are proper, lower semicontinuous  and \tn{e-$\lim_{i\to\infty}g_i=g$}. Let $V_i$ and $V$ be the value functions associated with $(\R^{n+1}\!\!\!,f_i,l_i,g_i)$ and $(\R^{n+1}\!\!\!,f,l,g)$, respectively.  Then for every $(t_0,x_0)\in[0,T]\times\R^n$ 
\begin{equation*}
\limsup\nolimits_{\,i\,\to\,\infty}V_i(t_{i0},x_{i0})\leq V(t_0,x_0)\;\;\;\it{for some sequence}\;\;\; (t_{i0},x_{i0})\rightarrow (t_0,x_0).
\end{equation*}
\end{Th}
\begin{proof}
Fix $(t_0,x_0)\in[0,T]\times\R^n$. Without loss of generality we may assume $V(t_0,x_0)<+\infty$ since otherwise there is nothing to prove. Then in view of Corollary \ref{cor-reduct} there exists the optimal arc $(\kr{x},\kr{a})(\cdot)$ of $V(t_0,x_0)$ defined on $[t_0,T]$. We extend $\kr{a}(\cdot)$ from $[t_0,T]$ to $[0,T]$ by setting $\kr{a}(t)=0$ for $t\in[0,t_0]$. Next, because of the sublinear growth of $f$, we extend  $\kr{x}(\cdot)$ from $[t_0,T]$ to $[0,T]$ such that $(\kr{x},\kr{a})(\cdot)\in\emph{S}_f(t_0,x_0)$. Because of \tn{e-$\lim_{i\to\infty}g_i=g$}, there exists a sequence $z_{i0}\to\kr{x}(T)$ such that $g_i(z_{i0})\to g(\kr{x}(T))$.
Now we choose $z_i(\cdot)$ defined on $[0,T]$ such that $(z_i,\kr{a})(\cdot)\in\emph{S}_{f_i}(T,z_{i0})$. Let $M>0$ be a constant such that $z_{i0},x_0,\kr{x}(T)\in\B_M$. Applying Gronwall’s Lemma  to  $(\kr{x},\kr{a})(\cdot)\in\emph{S}_f(T,\kr{x}(T))$ and $(z_i,\kr{a})(\cdot)\in\emph{S}_{f_i}(T,z_{i0})$, similarly as \eqref{rofvflem1} and \eqref{rofvflem2},  we obtain that $\|z_i\|\vee\|\kr{x}\|\leq R$ and
\begin{equation}\label{dn2-upperscvfs-1}
\|z_i-\kr{x}\|\;\leq\; \Big(|z_{i0}-\kr{x}(T)|+\|\,f_i[z_i]-f[z_i]\,\|_{L^1}\Big)\,\exp\Big(\int\limits_{\scriptscriptstyle 0}^{\scriptscriptstyle T}\omega_{\scriptscriptstyle R}[t]\,dt\Big).
\end{equation}
Similarly to  \eqref{dn-upperscvfs-3} we show that $\lim\nolimits_{i\to\infty}\|\,f_i[z_i]-f[z_i]\,\|_{L^1}=0$. The latter  and \eqref{dn2-upperscvfs-1}, together with $z_{i0}\to\kr{x}(T)$, imply that $\lim\nolimits_{i\to\infty}\|z_i-\kr{x}\|=0$. Hence we obtain $z_i(t_0)\to\kr{x}(t_0)=x_0$. Moreover,  similarly to   \eqref{dn-upperscvfs-4}, we can also show that $\lim\nolimits_{i\to\infty}\|\,l_i[z_i]-l[\kr{x}]\,\|_{L^1}=0$. Note that
\begin{eqnarray*}
V(t_0,x_0) &=& g(\kr{x}(T))+\int\limits_{\scriptscriptstyle t_0}^{\scriptscriptstyle T}l[\kr{x}](t)\,dt\\&\geq & g(\kr{x}(T))+\int\limits_{\scriptscriptstyle t_{0}}^{\scriptscriptstyle T}l_{i}[z_i](t)\,dt-\|\,l_i[z_i]-l[\kr{x}]\,\|_{L^1}\\[2mm]&\geq & g(\kr{x}(T))-g_i(z_{i0})+ V_i(t_{0},z_i(t_0))-\|\,l_i[z_i]-l[\kr{x}]\,\|_{L^1}.
\end{eqnarray*}
Passing to the limit in the above inequality, we get $\limsup\nolimits_{\,i\,\to\,\infty}V_i(t_0,z_i(t_0))\leq V(t_0,x_0)$, where $(t_0,z_i(t_0))\to(t_0,x_0)$.
\end{proof}

\begin{Rem}\label{rem-uscvvfs}
Let $(\R^{n+1}\!\!\!,f,l)$ be as above. Assume that $g$ is a continuous function. Let $V$ be the value function associated with $(\R^{n+1}\!\!\!,f,l,g)$. Applying Theorem \ref{d-upperscvfs-1} to $f_i:=f
\!\!$, $l_i:=l$, $g_i:=g$, we obtain that the value function $V$ is upper semicontinuous.
\end{Rem}

\subsection{Lower semicontinuity of value functions}\label{thm-reguvalufun-1}
Let $H,H_i,i\!\in\!\N$, satisfy \tn{(H1)-(H4)}\linebreak with the same  integrable maps $c(\cdot)$, $k_R(\cdot)$.   Assume that there exists integrable map $\mu_R(\cdot)$ such that $|H_i(t,x,0)|\vee|H(t,x,0)|\leq\mu_R(t)$ for all $(t,x)\in[0,T]\times\B_R$, $i\in\N$. We observe that by  the definition of conjugate we obtain $H_i^{\ast}(t,x,v)\wedge H^{\ast}(t,x,v)\geq-\mu_R(t)$ for all $t\in[0,T]$, $x\in\B_R$, $v\in\R^n$, $i\in\N$. Additionally, we assume that  $H_i(t,\cdot,\cdot)$ converge uniformly on compacts to $H(t,\cdot,\cdot)$ for all $t\in[0,T]$. This assumptions imply that the set-valued maps
\begin{align*}
Q(t,x)&:=\{\,(w,\eta)\in\R^n\times\R\mid (w,-\eta)\in\E H^{\ast}(t,x,\cdot)\,\}\\
Q_i(t,x)&:=\{\,(w,\eta)\in\R^n\times\R\mid (w,-\eta)\in\E H_i^{\ast}(t,x,\cdot)\,\}
\end{align*}
for all $t\in[0,T]$, $x\in\R^n$ have the following property (cf. Cesari \cite[Sections 8.5 and 10.5]{LC}) 
\begin{align}
& Q(t,x)=\bigcap_{\varepsilon\,>\,0}\,\overline{\mathrm{conv}}\,Q(t,x;\varepsilon), \; \tn{where}\label{wqcesari}\\
& Q(t,x;\varepsilon):=\bigcup_{1/i\;<\;\varepsilon,\;\;|x-y|\;<\;\varepsilon}Q_i(t,y).\nonumber
\end{align}

\begin{Lem}\label{wqclem}
Let  $H_i,H,Q_i,Q$, $i\in\N$ be as above. Assume that $(t_{i0},x_{i0},u_{i0})\to(t_0,x_0,u_0)$, $t_{i0},t_0\in[0,T)$. We consider a sequence of functions $(x_i,u_i)(\cdot)\in \mathcal{A}([t_{i0},T],\R^{n}\times\R)$ such that
\begin{equation}\label{wqclem-1}
(\dot{x}_i,\dot{u}_i)(t)\in Q_i(t,x_i(t))\;\;\it{a.e.}\;\;t\in[t_{i0},T],\;\;\;(x_i,u_i)(t_{i0})=(x_{i0},u_{i0}).
\end{equation}
Assume that $u_i(T)\geq M$ for every $i\in\N$ and some constant M. Then there exist   a function $(x,v)(\cdot)\in \mathcal{A}([t_{0},T],\R^{n}\times\R)$ and a real number $v_0\leq u_0$ such that
\begin{equation}\label{wqclem-2}
(\dot{x},\dot{v})(t)\in Q(t,x(t))\;\;\it{a.e.}\;\;t\in[t_{0},T],\;\;\;(x,v)(t_{0})=(x_{0},v_{0}).
\end{equation}
Moreover, there exist a subsequence $(x_{i_k},u_{i_k})(\cdot)$ of a sequence $(x_{i},u_{i})(\cdot)$ such that 
\begin{equation}\label{wqclem-3}
\lim\nolimits_{k\to\infty}x_{i_k}(T)=x(T)\;\quad\it{and}\quad\; \lim\nolimits_{k\to\infty}u_{i_k}(T)\leq v(T).
\end{equation}
\end{Lem}

\begin{proof}
Let $(t_{i0},x_{i0},u_{i0})\to(t_0,x_0,u_0)$. We consider a sequence of functions $(x_i,u_i)(\cdot)\in \mathcal{A}([t_{i0},T],\R^{n}\times\R)$ which satisfy \eqref{wqclem-1}, i.e.  $(x_i,u_i)(t_{i0})=(x_{i0},u_{i0})$ and
\begin{equation}\label{mfdszca-1}
-\dot{u}_i(t)\;\geq\; H_i^{\ast}(t,x_i(t),\dot{x}_i(t))\;\;\;\;\tn{a.e.}\;\;\;t\in[t_{i0},T].
\end{equation}
By our assumptions we can find a sequence $t_i\in[t_{i0},T]$ such that $t_i\to t_0^+$.

Therefore, $H_i^{\ast}(t,\kr{x}_i(t),\dot{\kr{x}}_i(t))<+\infty$ for a.e. $t\in[t_{i0},T]$ and all $i\in\N$. The latter, together with (C5), implies that $|\dot{x}_i(t)|\leq c(t)(1+|x_i(t)|)$ for a.e. $t\in[t_{i0},T]$ and all  $i\in\N$. Therefore, because of Gronwall’s Lemma, for every $i\in\N$,
\begin{equation*}
\|x_i(\cdot)\|\leq\left(\sup\nolimits_{i\in\N}\,|x_{i0}|\,+\int_0^Tc(t)\,dt\right)\,\exp\left(\int_0^Tc(t)\,dt\right)=:R.
\end{equation*} 
Hence $|\dot{x}_i(t)|\leq(1+R)\,c(t)$ for a.e. $t\in[t_{i0},T]$ and all $i\in\N$.  We observe that 
\begin{equation*}
|x_i(t_{i})-x_i(t_{i0})|\;\leq\;\int_{t_{i0}}^{t_i}|\dot{x}_i(t)|\;\leq\; (1+R)\int_{t_{i0}}^{t_i}c(t)\,dt\;\to\;0.
\end{equation*}
The latter, together with $x_i(t_{i0})=x_{i0}\to x_0$, implies that $x_i(t_i)\to x_0$. Now we extend $x_i(\cdot)$ from $[t_i,T]$ to $[t_0,T]$ setting $x_i(t)=x_i(t_i)$ for $t\in[t_0,t_i)$. We notice that the family $\{x_i(\cdot)\}_{i\in\N}$ of  such extended functions is equi-bounded. Moreover, the family $\{\dot{x}_i(\cdot)\}_{i\in\N}$ of derivatives is equi-absolutely integrable. Therefore, in view of Arzel\`a-Ascoli and Dunford-Pettis Theorems, there exists a subsequence (we do not relabel) such that $x_i(\cdot)$ converges uniformely to an absolutely continuous function $x:[t_0,T]\to\R^n$  and $\dot{x}_i(\cdot)$ converges weakly in $L^1([t_0,T],\R^n)$ to $\dot{x}(\cdot)$. We observe that $x(t_0)=x_0$.
 
 By our assumptions we conclude that $H_i^{\ast}(t,x_i(t),\dot{x}_i(t))\geq -\mu_R(t)$ for a.e. $t\in[t_{i0},T]$ and every $i\in\N$. The latter, together with \eqref{mfdszca-1}, implies that $\dot{u}_i(t)\leq\mu_R(t)$ for a.e. $t\in[t_{i0},T]$ and every $i\in\N$. Moreover, we know that $u_i(T)\geq M$ for every $i\in\N$. Therefore,
\begin{align*}
u_i(t) &\;=\;u_i(T)-\int_t^T\dot{u}_i(s)\,ds\;\geq\; M-\int_0^T\mu_R(s)\,ds,\\
u_i(t) &\;=\;u_i(t_{i0})+\int_{t_{i0}}^t\dot{u}_i(s)\,ds\;\leq\;\sup\nolimits_{i\in\N}|u_{i0}|+\int_0^T\mu_R(s)\,ds,
\end{align*}
for all $t\in[t_{i0},T]$ and $i\in\N$. Since  $\{u_i(t_i)\}_{i\in\N}$ is bounded, we conclude that there exists a subsequence (we do not relabel) such that $u_i(t_i)\to v_0$. We observe that for all $i\in\N$
$$u_i(t_i)\;=\;u_i(t_{i0})+\int_{t_{i0}}^{t_i}\dot{u}_i(t)\,dt\leq u_{i0}+\int_{t_{i0}}^{t_i}\mu_R(t)\,dt.$$
The latter, together with $u_{i0}\to u_0$, implies that $v_0\leq u_0$.
Furthermore, we observe that
\begin{eqnarray*}\label{mzfw-Lem1-n2}
\mathrm{Var}_{[t_{i0},T]}u_i(\cdot) & \leq & \int_{t_{i0}}^T|\dot{u}_i(t)|\,dt\;\;=\;\; \int_{t_{i0}}^T2\,(\,\dot{u}_i(t)\vee 0\,)\,dt-\int_{t_{i0}}^T\dot{u}_i(t)\,dt\\
& \leq & \int_0^T2\;\mu_R(t)\,dt\;+\; u_i(t_{i0})\;-\;u_i(T)\\
&\leq & \int_0^T2\;\mu_R(t)\,dt\;+\;\sup\nolimits_{i\in\N}|u_{i0}|\;+\;|M|.
\end{eqnarray*}
Now we extend $u_i(\cdot)$ from $[t_i,T]$ to $[t_0,T]$ setting $u_i(t)=u_i(t_i)$ for $t\in[t_0,t_i)$. We notice that the family $\{u_i(\cdot)\}_{i\in\N}$ of  such extended functions is equi-bounded. Moreover, the variations of functions $u_i(\cdot)$ on $[t_0,T]$ are equi-bounded. Therefore, in view of Helly Theorem, there exists a subsequence (we do not relabel) such that $u_i(\cdot)$ converges pointwise (everywhere) to a bounded variation function $u:[t_0,T]\to\R$.

Since $\dot{x}_{s+k}(\cdot)\rightharpoonup\dot{x}(\cdot)$ as $k\to\infty$ in $L^1([t_0,T],\R^n)$ for all $s\in\N$, by the Mazur Theorem, there exist real numbers $\lambda_{k,N}^{s}\geq 0$ for  $k=1,2,\ldots,N$ and $N\in\N$  such that $\sum_{k=1}^{N}\lambda_{k,N}^{s}=1$
and
$\sum_{k=1}^{N}\lambda_{k,N}^{s}\dot{x}_{s+k}(\cdot)\longrightarrow\dot{x}(\cdot)$ as $N\to\infty$
in $L^1([t_0,T],\R^n)$ for all $s\in\N$. Then for all $s\in\N$ there exists an increasing sequence  $\{N_{i}^s\}_{i\in\N}$ such that 
\begin{equation*}
z_{i}^{s}(t):=\sum_{k=1}^{N_{i}^s}\lambda_{k,N_{i}^s}^{s}\dot{x}_{s+k}(t)\longrightarrow\dot{x}(t)\;\;\;\tn{as}\;\;\;i\to\infty\;\;\;\;
\tn{a.e.}\;\;\;t\in[t_{0},T]. \label{g1}
\end{equation*} 
For a.e. $t\in[t_0,T]$ we define functions:
$$\eta_i(t):=\mathbbm{1}_{[t_i,T]}(t)\,H_i^{\ast}(t,x_i(t),\dot{x}_i(t)),\quad\eta_{i}^{s}(t):=
\sum_{k=1}^{N_{i}^s}\lambda_{k,N_i^s}^{s}\eta_{s+k}(t),$$
where  $\mathbbm{1}_{[t_i,T]}(\cdot)$ has value $1$ on $[t_i,T]$ but $0$ outside, and 
$$\eta^{s}(t):=\liminf\nolimits_{i\to\infty}\eta_{i}^{s}(t),\qquad\eta(t):=\liminf\nolimits_{s\to\infty}\eta^{s}(t). $$
We observe that $(\,\eta_{i}^{s}(t)\wedge\eta^{s}(t)\wedge\eta(t)\,)\geq -\mu_R(t)$ for a.e. $t\in[t_0,T]$ and all $s,i\in\N$.

Now we show that for all $\tau_0\in[t_0,T]$ the following inequality is true:
\begin{equation}\label{nc11}
\int_{t_0}^{\tau_0}\eta(t)\:dt\leq u(t_0)-u(\tau_0).
\end{equation}
Indeed, fix $\tau_0\in(t_0,T]$ and $\varepsilon>0$. We find $s_0\in\N$ such that $u_{s+k}(t_{s+k})=u_{s+k}(t_0)\leq u(t_0)+\varepsilon$,
$u(\tau_0)-\varepsilon\leq u_{s+k}(\tau_0)$, $t_{s+k}<\tau_0$ for all $s>s_0$, $k\in\mathbb{N}$. 
 Then for all $i\in\mathbb{N}$ and  $s>s_0$ we have
\begin{eqnarray*}
\int_{t_0}^{\tau_0}\eta_i^s(t)\,dt & = & \sum_{k=1}^{N_i^s}\lambda_{k,N_i^s}^s\int_{t_0}^{\tau_0}\eta_{s+k}(t)\,dt\;\;=\;\;\sum_{k=1}^{N_i^s}\lambda_{k,N_i^s}^s\int_{t_{s+k}}^{\tau_0}\eta_{s+k}(t)\,dt\\
& \leq & \sum_{k=1}^{N_i^s}\lambda_{k,N_i^s}^s\int_{t_{s+k}}^{\tau_0}\!\!-\dot{u}_{s+k}(t)\,dt
\;\;=\;\; \sum_{k=1}^{N_i^s}\lambda_{k,N_i^s}^s(u_{s+k}(t_{s+k})-u_{s+k}(\tau_0))\\
& \leq & \sum_{k=1}^{N_i^s}\lambda_{k,N_i^s}^s(u(t_0)-u(\tau_0)+2\varepsilon)
\;\;=\;\; u(t_0)-u(\tau_0)+2\varepsilon.
\end{eqnarray*}
By the Fatou Lemma we have
\begin{eqnarray*}
\int_{t_0}^{\tau_0}\eta(t)\:dt & \leq & \liminf_{s\to\infty} \int_{t_0}^{\tau_0}\eta^s(t)\:dt\\
 & \leq & \liminf_{s\to\infty} \liminf_{i\to\infty}\int_{t_0}^{\tau_0}\eta_i^s(t)\:dt\\
& \leq & u(t_0)-u(\tau_0)+2\varepsilon.
\end{eqnarray*}
Since   $\varepsilon>0$ is arbitrary, we conclude that the inequality \eqref{nc11} is true.

Now we show that for a.e. $t\in[t_0,T]$ we have 
\begin{equation}\label{ink11}
(\dot{x}(t),-\eta(t))\in Q(t,x(t)).
\end{equation}
Indeed, let us fix $t\in(t_0,T)$ and $\varepsilon>0$. Then for all large $s\in\N$ and all $k\in\N$ we have $|x_{s+k}(t)-x(t)|\leq\varepsilon$. Moreover, 
for all large $i\in\mathbb{N}$ we have
$$(\dot{x}_i(t),-\eta_i(t))\in Q_i(t,x_i(t)).$$
For all large $s\in\mathbb{N}$  and all $k\in\mathbb{N}$ we have
$$(\dot{x}_{s+k}(t),-\eta_{s+k}(t))\in Q(t,x(t);\varepsilon),$$
for all large $s\in\mathbb{N}$  and all $i\in\mathbb{N}$ we have
$$(z_i^s(t),-\eta_i^s(t))\in \mathrm{conv}\,Q(t,x(t);\varepsilon).$$
Passing to a subsequence as $i\rightarrow\infty$ we have
$$(\dot{x}(t),-\eta^s(t))\in \overline{\mathrm{conv}}\,Q(t,x(t);\varepsilon),$$
passing to a subsequence as $s\rightarrow\infty$ we have
$$(\dot{x}(t),-\eta(t))\in \overline{\mathrm{conv}}\,Q(t,x(t);\varepsilon).$$
Since   $\varepsilon>0$ is arbitrary, we obtain
$$(\dot{x}(t),-\eta(t))\in \bigcap_{\varepsilon>0}\overline{\mathrm{conv}}\, 
Q(t,x(t);\varepsilon).$$
The latter, together with \eqref{wqcesari}, implies \eqref{ink11}.

Let $v(\tau):=u(t_0)-\int_{t_0}^\tau\eta(t)\,dt$. Then $v(\cdot)$ is an absolutely continuous function on $[t_0,T]$. Moreover, in view of \eqref{nc11},  we have $v(t)\geq u(t)$ for all $t\in[t_0,T]$. Furthermore, in view of \eqref{ink11}, we have $(\dot{x}(t),\dot{v}(t))\in Q(t,x(t))$ for a.e. $t\in[t_0,T]$.

We conclude that $(x,v)(\cdot)$ is an absolutely continuous function defined on $[t_0,T]$ which satisfies \eqref{wqclem-2} and $v_0\leq u_0$. Moreover, $(x_i,u_i)(\cdot)$ converges pointwise to $(x,u)(\cdot)$ along a subsequence. The latter, together with $v(\cdot)\geq u(\cdot)$, implies \eqref{wqclem-3}.
\end{proof}

\begin{Th}\label{thm-lscvvfs-1}
Let  $H_i,H$, $i\in\N$ be as above. Assume that $g_i$ and $g$ are proper, lower semicontinuous functions and \tn{e-$\lim_{i\to\infty}g_i=g$}. Let $V_i$ and $V$ be the value functions associated with $(H_i^{\ast},g_i)$ and $(H^{\ast},g)$, respectively.  Then for every $(t_0,x_0)\in[0,T]\times\R^n$ 
\begin{equation*}
\liminf\nolimits_{\,i\,\to\,\infty}V_i(t_{i0},x_{i0})\geq V(t_0,x_0)\;\;\;\it{for every sequence}\;\;\; (t_{i0},x_{i0})\rightarrow (t_0,x_0).
\end{equation*}
\end{Th}
\begin{proof}
Fix $(t_{i0},x_{i0})\rightarrow (t_0,x_0)$. Let $\Delta:=\liminf\nolimits_{\,i\,\to\,\infty}V_i(t_{i0},x_{i0})$. We show that $V(t_0,x_0)\leq\Delta$. Without loss of generality we may assume $t_{i0}<T$ and $\Delta<+\infty$. Because of the definition $\Delta$, there exists  a subsequence
(we do not relabel) such that $V_i(t_{i0},x_{i0})\to \Delta$. Hence for all large $i\in\N$ we have $V_i(t_{i0},x_{i0})<+\infty$. In view of Corollary \ref{cor-reduct} there exist a sequence of functions $x_i(\cdot)\in\mathcal{A}([t_{i0},T],\R^{n})$ such that $x_i(t_{i0})=x_{i0}$ and
\begin{equation}\label{cifwpwd}
V_i(t_{i0},x_{i0})\;=\;g_i(x_i(T))+\int_{t_{i0}}^TH_i^{\ast}(t,x_i(t),\dot{x}_i(t))\,dt.
\end{equation}
Since $V_i(t_{i0},x_{i0})<+\infty$ for all large $i\in\N$, we have  $H_i^{\ast}(t,x_i(t),\dot{x}_i(t))<+\infty$ for a.e. $t\in[t_{i0},T]$ and all large $i\in\N$. The latter, together with (C5), implies that $|\dot{x}_i(t)|\leq c(t)(1+|x_i(t)|)$ for a.e. $t\in[t_{i0},T]$ and all large $i\in\N$. Thus, because of Gronwall’s Lemma, for all large $i\in\N$,
\begin{equation*}
\|x_i(\cdot)\|\leq\left(\sup\nolimits_{i\in\N}\,|x_{i0}|\,+\int_0^Tc(t)\,dt\right)\,\exp\left(\int_0^Tc(t)\,dt\right)=:R.
\end{equation*} 
Hence $|\dot{x}_i(t)|\leq(1+R)\,c(t)$ for a.e. $t\in[t_{i0},T]$ and all large $i\in\N$. By our assumptions we conclude that $H_i^{\ast}(t,x_i(t),\dot{x}_i(t))\geq -\mu_R(t)$ for a.e. $t\in[t_{i0},T]$ and all large $i\in\N$. Because of \tn{e-$\lim_{i\to\infty}g_i=g$}, there exists a constant $M$ such that $g_i(x)\geq M$ for all $x\in\B_R$ and all large $i\in\N$. Therefore, in view of \eqref{cifwpwd}, for all large $i\in\N$, we have 
\begin{equation*}
V_i(t_{i0},x_{i0})\;\geq\;M-\int_0^T\mu_R(t)\,dt \;>\;-\infty.
\end{equation*}
Hence $\Delta>-\infty$.  To prove theorem  we consider two cases:

\bf{Case 1.} Let us consider $t_0<T$. We put $u_{i0}:=V_i(t_{i0},x_{i0})$ for all large $i\in\N$ and $u_0:=\Delta$. We define a sequence of functions $u_i(\cdot)\in\mathcal{A}([t_{i0},T],\R)$ for all large $i\in\N$ by 
\begin{equation*}
u_i(t)\;=\;g_i(x_i(T))+\int_{t}^TH_i^{\ast}(s,x_i(s),\dot{x}_i(s))\,ds.
\end{equation*}
We observe that $-\dot{u}_i(t)=H_i^{\ast}(t,x_i(t),\dot{x}_i(t))$ for a.e. $t\in[t_{i0},T]$ and 
$u_i(t_{i0})=V_i(t_{i0},x_{i0})=u_{i0}$ for all large $i\in\N$. It means that the sequence  $(x_i,u_i)(\cdot)$ satisfies \eqref{wqclem-1} for all large $i\in\N$. Moreover, $u_i(T)=g_i(x_i(T))\geq M$ for all large $i\in\N$. Therefore, because of Lemma \ref{wqclem}, there exist   a function $(x,v)(\cdot)\in \mathcal{A}([t_{0},T],\R^{n}\times\R)$ and a real number $v_0\leq u_0$ such that \eqref{wqclem-2} holds.
In view of \eqref{wqclem-2}, we deduce that
\begin{eqnarray}
\Delta \;\;=\;\; u_0 &\geq & v_0\;\;=\;\;v(t_0)\;\;=\;\;v(T)+\int_{t_0}^T\!\!-\dot{v}(t)\,dt\nonumber\\[-3mm]
&\geq & v(T)+\int_{t_0}^TH^{\ast}(t,x(t),\dot{x}(t))\,dt\nonumber\\
&\geq & v(T)-g(x(T))+V(t_0,x_0).\label{deltacr1}
\end{eqnarray}
Moreover, in view of \eqref{wqclem-3} and \tn{e-$\lim_{i\to\infty}g_i=g$}, we deduce that
\begin{equation}\label{deltacr2}
v(T)\;\geq\; \lim_{k\to\infty}u_{i_k}(T)\;=\;\lim_{k\to\infty}g_{i_k}(x_{i_k}(T))\;\geq\; g(x(T)).
\end{equation}
Combining inequalities \eqref{deltacr1} and \eqref{deltacr2} we obtain $\Delta\geq V(t_0,x_0)$.

\bf{Case 2.} Let us consider $t_0=T$. We observe that
\begin{equation*}
|x_i(t_{i0})-x_i(T)|\;\leq\;\int_{t_{i0}}^T|\dot{x}_i(t)|\;\leq\; (1+R)\int_{t_{i0}}^Tc(t)\,dt\;\to\;0.
\end{equation*}
The latter, together with $x_i(t_{i0})=x_{i0}\to x_0$, implies that $x_i(T)\to x_0$. Therefore, in view of \eqref{cifwpwd} and \tn{e-$\lim_{i\to\infty}g_i=g$}, we have 
\begin{eqnarray*}
\Delta=\lim_{i\to\infty}V_i(t_{i0},x_{i0}) &\geq & \liminf_{i\to\infty} g_i(x_i(T))+\liminf_{i\to\infty}\int_{t_{i0}}^TH_i^{\ast}(t,x_i(t),\dot{x}_i(t))\,dt\\
&\geq & \liminf_{i\to\infty} g_i(x_i(T))+\liminf_{i\to\infty}\int_{t_{i0}}^T\!\!-\mu_R(t)\,dt\\[2mm]
&\geq & g(x_0)\,=\,V(T,x_0)\,=\,V(t_0,x_0).
\end{eqnarray*}
This completes the proof of the theorem.
\end{proof}

\begin{Rem}\label{rem-lscvvfs}
Consider  $H$ as above. Assume that $g$ is a proper, lower semicontinuous function. Let $V$ be the value function associated with $(H^{\ast},g)$. Applying Theorem \ref{thm-lscvvfs-1} to $H_i:=H$, $g_i:=g$, we obtain that the value function $V$ is lower semicontinuous.
\end{Rem}

\subsection{Remarks}\label{rem-conclude}
In the proof of stability of value functions we used the formula~\eqref{fwfl} on the value function as well as the formula \eqref{fwwp}. We do that, because formulas \eqref{fwwp} and~\eqref{fwfl} have  advantages and drawbacks.

The advantages of the formula \eqref{fwfl} are regularities of functions
$f$ and $l$ such that: a sublinear growth of the function $f$ with respect to the state variable, a sublinear growth of the function $l$ with respect to the control variable and  local Lipschitz continuity with respect to the state variable for both functions $f$ and $l$. These regularities of functions $f$ and $l$ together with the extra-property allow us to prove upper semicontinuity of  value functions. On the other hand, the problems appear in the proof of  lower semicontinuity of value functions. They can be overcome using convexity and coercivity of the function $l$ with respect to the control variable; see \cite{B-DM,CO-76}. However, in our case the function  $l$ does not possess these properties and it is a drawback of the formula \eqref{fwfl}.   

Lower semicontinuity of  value functions is proven using the formula \eqref{fwwp}. 
It is \linebreak possible due to convexity and coercivity of the conjugate $H^{\ast}(t,x,\cdot)$. These properties of the conjugate $H^{\ast}$  are advantages of the formula  \eqref{fwwp}. The example of the Hamiltonian $H$\linebreak in Section \ref{dfrentger} shows that the  conjugate $H^{\ast}$ is an  extended-real-valued function and\linebreak discontinuous on the effective domain $\D H^{\ast}$.  These 
properties of the conjugate $H^{\ast}$  are drawbacks of the formula \eqref{fwwp}.

\subsection{Lipschitz continuity of the value function}\label{thm-reguvalufun-0} Assume that \tn{(H1)-(H4)} and \tn{(HLC)} hold with integrable functions $c(\cdot)$, $k_R(\cdot)$, $H(\cdot,0,0)$.  Let $g$ be a locally Lipschitz function.  We consider the representation $(\R^{n+1}\!\!\!,f,l)$ of $H$ defined as in  Theorem~\ref{th-rprez-glo12}. Let $M>0$ and
\begin{align*}
 R&:=\left(M+\int_0^Tc(t)\,dt\right)\,\exp\left(\int_0^Tc(t)\,dt\right),\\
C_M&:=\left(D_R+\int_0^T\omega_{\scriptscriptstyle R}[t]\,dt\right)\,\exp\left(\int_0^T\omega_{\scriptscriptstyle R}[t]\,dt\right),
\end{align*}
where $\omega_{\scriptscriptstyle R}[\cdot]=2|H(\cdot,0,0)|+(10\,(n+1)\,k_R(\cdot)+2c(\cdot))(1+R)$ and $D_R$ denotes the Lipschitz constant of $g$ on $B_R$. Let us consider the following function
\begin{equation*}
\alpha_{M}(t):=(1+C_{M})\int_0^t\omega_{\scriptscriptstyle R}[s]\,ds\;\;\;\tn{for all}\;\;\;t\in[0,T].
\end{equation*}

\begin{Prop}\label{thmncfop3}
Consider $(\R^{n+1}\!\!\!,f,l)$ as above. Assume that $g$ is a real-valued lower semicontinuous function. If  $V$ is the value function associated with  $(\R^{n+1}\!\!\!,f,l,g)$, then $V$ is a real-valued function on  $[0,T]\times\R^n$.
\end{Prop}
\begin{proof}
Fix $t_0\in[0,T]$ and $x_0\in\R^n$. We show that $-\infty< V(t_0,x_0)<+\infty$. Observe that the first inequality  follows from \eqref{roz6-l2}. It remains to prove the second inequality.  Let $\tilde{a}(\cdot)\equiv 0$. Then there exists $\tilde{x}(\cdot)\in \mathcal{A}([t_0,T],\R^n)$ such that $\dot{\tilde{x}}(t)=f(t,\tilde{x}(t),\tilde{a}(t))$ for a.e. $t\in [t_0, T]$ and $x(t_0)=x_0$.  In view of Theorem~\ref{th-rprez-glo12} (A2) and  (HLC)  we get that
\begin{eqnarray*}
\nonumber l(t,\tilde{x}(t),\tilde{a}(t)) &\leq & 2|H(t,\tilde{x}(t),0)|+2c(t)(1+|\tilde{x}(t)|)+3|\tilde{a}(t)|\\
&\leq & 2|H(t,0,0)|+2k_{\|\tilde{x}\|}(t)\,\|\tilde{x}\|+2c(t)(1+\|\tilde{x}\|)=:\mu(t)
\end{eqnarray*}
for a.e. $t\in [t_0, T]$.  Thus, $V(t_0,x_0)\leq g(\tilde{x}(T))+\|\mu\|_{L^1}<+\infty$.
\end{proof}

\begin{proof}[Proof of Theorem \ref{lip-value-function-2}]
Fix $t_0,s_0\in[0,T]$ and $x_0,y_0\in\B_M$. Then, by Corollary \ref{cor-reduct} there exists the optimal arc $(\bar{x},\bar{a})(\cdot)$ of $V(t_0,x_0)$ defined on $[t_0,T]$. We extend $\bar{a}(\cdot)$ from $[t_0,T]$ to $[0,T]$ by setting $\bar{a}(t)=0$ for $t\in[0,t_0]$. Next, because of the sublinear growth of $f$, we extend  $\bar{x}(\cdot)$ from $[t_0,T]$ to $[0,T]$ such that $(\bar{x},\bar{a})(\cdot)\in\emph{S}_f(t_0,x_0)$. Now we choose $y(\cdot)$ defined on $[0,T]$ such that $(y,\bar{a})(\cdot)\in\emph{S}_f(s_0,y_0)$. By Gronwall's Lemma we get $ \|\bar{x}\|\vee\|y\|\leq R$,
\begin{align}
&\|\bar{x}-y\|\leq\Big(|x_0-y_0|+\int\limits_{\scriptscriptstyle t_0\wedge s_0}^{\scriptscriptstyle t_0\vee s_0}\omega_{\scriptscriptstyle R}[t]\,dt\Big)\,\exp\Big(\int\limits_{\scriptscriptstyle 0}^{\scriptscriptstyle T}\omega_{\scriptscriptstyle R}[t]\,dt\Big),\label{rofvflem11}\\
&\int\limits_{\scriptscriptstyle  t_0\wedge s_0}^{\scriptscriptstyle T}\big|l[\bar{x}](t)-l[y](t)\big|\,dt\;\leq \|\bar{x}-y\|\int\limits_{\scriptscriptstyle 0}^{\scriptscriptstyle T}\omega_{\scriptscriptstyle R}[t]\,dt,\label{rofvflem22}
\end{align}
where $l[x](\cdot):=l(\cdot,x(\cdot),\bar{a}(\cdot))$. To prove theorem  we consider two cases:

\bf{Case 1.} Let $t_0\leq s_0$. By  Theorem~\ref{th-rprez-glo12} (A2) we have $l[\bar{x}](t)\geq-\omega_{\scriptscriptstyle R}[t]$ for all $t\in[0,T]$. The latter, together with \eqref{rofvflem11} and \eqref{rofvflem22}, implies 
\begin{eqnarray*}
V(s_0,y_0)-V(t_0,x_0) &\leq & g(y(T))+\int\limits_{\scriptscriptstyle s_0}^{\scriptscriptstyle T}l[y](t)\,dt-g(\bar{x}(T))-\int\limits_{\scriptscriptstyle t_0}^{\scriptscriptstyle T}l[\bar{x}](t)\,dt\\
&\leq & |g(\bar{x}(T))-g(y(T))|+ \int\limits_{\scriptscriptstyle  s_0}^{\scriptscriptstyle T}\big|l[\bar{x}](t)-l[y](t)\big|\,dt-\int\limits_{\scriptscriptstyle t_0}^{\scriptscriptstyle s_0}l[\bar{x}](t)\,dt\\
&\leq & \|\bar{x}-y\|\Big(D_R+\int\limits_{\scriptscriptstyle 0}^{\scriptscriptstyle T}\omega_{\scriptscriptstyle R}[t]\,dt\Big)+ \int\limits_{\scriptscriptstyle t_0}^{\scriptscriptstyle s_{0}}\omega_{\scriptscriptstyle R}[t]\,dt\\
&\leq & C_M|x_0-y_0|+(1+C_M)\int\limits_{\scriptscriptstyle t_{0}}^{\scriptscriptstyle s_0}\omega_{\scriptscriptstyle R}[t]\,dt\\[1mm]
&=& C_M|x_0-y_0|+ |\alpha_M(t_0)-\alpha_M(s_0)|.
\end{eqnarray*}

\bf{Case 2.} Let $s_0\leq t_0$. Then  $\bar{a}(t)=0$ for all $t\in[s_0,t_0]$. By Theorem~\ref{th-rprez-glo12} (A2) we have $l[\bar{x}](t)\leq\omega_{\scriptscriptstyle R}[t]+3|\bar{a}(t)|$ for all $t\in[0,T]$.  The latter, together with \eqref{rofvflem11} and \eqref{rofvflem22}, implies   
\begin{eqnarray*}
V(s_0,y_0)-V(t_0,x_0) &\leq & g(y(T))+\int\limits_{\scriptscriptstyle s_0}^{\scriptscriptstyle T}l[y](t)\,dt-g(\bar{x}(T))-\int\limits_{\scriptscriptstyle t_0}^{\scriptscriptstyle T}l[\bar{x}](t)\,dt\\
&\leq & |g(\bar{x}(T))-g(y(T))|+\int\limits_{\scriptscriptstyle s_0}^{\scriptscriptstyle t_0}l[\bar{x}](t)\,dt+ \int\limits_{\scriptscriptstyle  t_0}^{\scriptscriptstyle T}\big|l[\bar{x}](t)-l[y](t)\big|\,dt\\
&\leq & \|\bar{x}-y\|\,\Big(D_R+\int\limits_{\scriptscriptstyle 0}^{\scriptscriptstyle T}\omega_{\scriptscriptstyle R}[t]\,dt\Big)+ \int\limits_{\scriptscriptstyle s_0}^{\scriptscriptstyle t_{0}}\omega_{\scriptscriptstyle R}[t]\,dt+3\int\limits_{\scriptscriptstyle s_0}^{\scriptscriptstyle t_{0}}|\bar{a}(t)|\,dt
\end{eqnarray*}

\vspace*{-5mm}
\pagebreak
\begin{eqnarray*}
\hspace{1.8cm}&\leq & C_M|x_0-y_0|+(1+C_M)\int\limits_{\scriptscriptstyle s_{0}}^{\scriptscriptstyle t_0}\omega_{\scriptscriptstyle R}[t]\,dt+3\int\limits_{\scriptscriptstyle s_0}^{\scriptscriptstyle t_{0}}|\bar{a}(t)|\,dt\\[1mm]
&=& C_M|x_0-y_0|+ |\alpha_M(t_0)-\alpha_M(s_0)|.
\end{eqnarray*}

In view of  Case 1 and Case 2, we conclude that the inequality \eqref{nlfw-12} is true. If $c(\cdot)$, $k_R(\cdot)$, $H$ are continuous, so is $\omega_{\scriptscriptstyle R}(\cdot)$. In this case, we show that $V$  is Lipschitz continuous on $[0,T]\!\times\!\B_R$. Because of  \eqref{nlfw-12}, it suffices to note that  $|\alpha_M(t_0)\!-\!\alpha_M(s_0)|\leq (1\!+\!C_M)\|\omega_R\||t_0\!-\!s_0|$ for all $t_0,s_0\in[0,T]$, $x_0,y_0\in\B_M$. This completes the proof of the theorem.
\end{proof}


\section{Concluding remarks}\label{conrem-conclude}

\noindent In the case of  representations with a compact control set, we knew what type of\linebreak regularity we should expect, because this kind of representations had been considered by Frankowska-Sedrakyan and  Rampazzo. Moreover,  we found their broad applications in the monograph Bardi and Capuzzo-Dolcetta.  However, in Theorem \ref{th-rprez-glo12} the case of regularities of representations with the unbounded control set is much more complicated.  Rampazzo-Sartori applied such representations in studies on regularity of value functions. However, they assumed  coercivity of the function  $l(t,x,\cdot)$. Unfortunately, the function $l(t,x,\cdot)$ from our faithful representation does  not have this property. Therefore, their proofs cannot be applied  in our case. This problem has been solved due to the extra-property and upper-boundedness of the  function $l$. Investigating applications of  representations with the unbounded control set leads us to the fundamental relation between variational and optimal control problems; see Theorem \ref{thm-reduct}. The correlation between variational and\linebreak optimal control problems has not been used earlier. For the first time we have used this  correlation in the proof of  Theorem  \ref{cor-rep-stab4}. Significant differences between representations\linebreak with compact and unbounded control sets triggered us to write two distinct papers\linebreak containing results related to them.




\end{document}